\documentclass[a4paper]{amsart}
\usepackage{amsmath,amsthm,amssymb,latexsym,epic,bbm,comment,color}
\usepackage{graphicx,enumerate,stmaryrd}
\usepackage[all,2cell]{xy}
\xyoption{2cell}

\newtheorem{theorem}{Theorem}[section]
\newtheorem{lemma}[theorem]{Lemma}
\newtheorem{corollary}[theorem]{Corollary}
\newtheorem{proposition}[theorem]{Proposition}

\newtheorem{example}[theorem]{Example}

\usepackage[all]{xy}
\usepackage[active]{srcltx}
\usepackage[parfill]{parskip}

\newcommand{\tto}{\twoheadrightarrow}
\font\sc=rsfs10
\newcommand{\cC}{\sc\mbox{C}\hspace{1.0pt}}

\begin{document}
\title[Extreme representations of semirings]
{Extreme representations of semirings}

\author[Chen-Dubsky-Jonsson-Mazorchuk-Persson Westin-Zhang-Zimmermann]{Chih-Whi Chen, 
Brendan Frisk Dubsky, Helena Jonsson,\\ Volodymyr Mazorchuk, Elin Persson Westin,\\ 
Xiaoting Zhang, Jakob Zimmermann}

\begin{abstract}
This is a write-up of the discussions during the meetings of the study group on 
representation theory of semirings which was organized at the Department of Mathematics, 
Uppsala University, during  the academic year 2017-2018. The main emphasis is
on classification of various classes of ``irreducible'' representations for
various concrete semirings. 
\end{abstract}

\maketitle

\section{Introduction}\label{s1}

Abstract structure and representation theory of semirings managed to successfully stay away from 
mathematical mainstream over the years. This is despite of the fact that there are monographs,
like e.g. \cite{Go,GM,Wi}, devoted to it. One of the possible reasons for this might be that it is
significantly more complicated than the classical structure and representation theory
of rings.

The present paper is a write-up of the discussions during the meetings of the study group on 
representation theory of semirings which was organized at the Department of Mathematics, 
Uppsala University, during  the academic year 2017-2018. Our interest in the subject stems
from its connection to higher representation theory and categorification, see \cite{Ma1,Ma2}.
A typical object of study in higher representation theory is a Krull-Schmidt tensor category 
$\cC$ with finitely many isomorphism classes of indecomposable objects. The
{\em Grothendieck decategorification} of $\cC$ is the split Grothendieck group 
$[\cC]_{\oplus}$ of $\cC$
which carries a natural structure of a $\mathbb{Z}$-algebra. The following question 
sounds natural in this context:

{\em What will one gain or loose by looking instead of $[\cC]_{\oplus}$
at the natural semiring structure on the set of isomorphism classes of 
objects in $\cC$?}

However, before we can even start thinking about this question, one needs to learn
a little bit about semirings and their representations. That was the aim of our study group.
It seems that even the basic terminology of the theory is not established yet
(for example, semimodules in \cite{Go} are called modules in \cite{Wi}). Therefore
we tried, when we thought appropriate, to come up with a ``better'' alternative.
One such case is the notion of a {\em simple} semimodule. In ring theory, simple modules
are ``smallest possible'' both with respect to taking submodules and quotients
(as the latter two requirements are equivalent). For semirings, the notion of
``smallest possible'' semimodule with respect to taking submodules is not equivalent,
in general, to the notion of a ``smallest possible'' semimodule with respect to taking 
quotients. In \cite{Go}, the former are called {\em minimal} while the latter 
are called {\em simple}. The second term
might be motivated by the notion of a simple ring, but is really
confusing from the point of view of module theory. Therefore we propose to 
call ``smallest possible'' semimodules with respect to taking 
quotients {\em elementary} and keep the word {\em simple} to describe those semimodules
which are both minimal and elementary at the same time.   Semimodules which are 
either minimal or elementary (or both) are called {\em extreme}. 

The main emphasis of the text is on the classification of (some classes of) extreme
semimodules for various concrete semirings (including, in particular,
the Boolean semiring and the semiring of all non-negative integers). Our aim was to 
look at examples to see whether a solution to the classification problem
in these examples seems possible and  how all these different ``simplicity''
notions can be different in examples. One of the motivating examples
was the $\mathbb{Z}_{\geq 0}$-semiring generated by the elements of a Kazhdan-Lusztig
basis in the integral group ring of a finite Coxeter system, see \cite{KL,EW}.
For such a semiring, we give a complete solution for classification of all extreme 
semimodules in types $A_1$, $A_2$ and of all proper extreme
semimodules in all dihedral types, see Sections~\ref{s5}, \ref{s7}
and \ref{s9}, respectively.
Taking into account the answer in type $A_2$,  one might expect that,
in the general case, a complete classification might be rather hard.
We get, however, some general results on minimal semimodules in Sections~\ref{s8}.

Along the way, we solve the classification problem for the Boolean semiring,
for the semiring of all non-negative integers and the group semiring of a
finite group over the latter. It might well be that some of these results 
are already known and can be found in the literature. We have not seen them
and the fact that most of these results are easier and faster to prove
directly than to look in the literature is strongly discouraging from
spending too much time on looking. We apologize if in this way we missed
some references and are happy to add them in the revised version if we get 
any hints about them.

Apart from that, the text mainly follows the time line of the discussions
during our meetings. One could suggest that the paper could be organized
more efficiently by combining some of the results and that some directions 
described in this manuscript could be developed further, but, unfortunately,
this is not possible due to the time constraints on the present format of 
this study group. Most significant (seemingly) original results 
are in Sections~\ref{s7}, \ref{s8} and \ref{s9}. 

We start in Section~\ref{s2} with a description of the setup and
basic terminology. Section~\ref{s25} studies extreme semimodules over
the Boolean semiring and its various generalizations.
Section~\ref{s3} considers extreme semimodules over the semiring of 
non-negative integers. Section~\ref{s5} studies extreme semimodules
in the case of the group semiring over non-negative integers
of the $2$-element group and also its subsemiring corresponding 
to the Kazhdan-Lusztig basis.
In Section~\ref{s6} we consider the semiring of non-negative real numbers.
Unlike the previous cases, here we really see for the first time
how  the notions of minimal and elementary semimodules can be different.
Section~\ref{s4} contains some general results like an analogue of
Schur's lemma and some detailed general information on the structure of
minimal and elementary proper semimodules. Section~\ref{s7}
classifies all types of extreme semimodules for the 
group semiring of the symmetric group $S_3$ over
$\mathbb{Z}_{\geq 0}$ in the Kazhdan-Lusztig basis.
Section~\ref{s8} studies finitely generated 
$\mathbb{Z}_{\geq 0}$-semirings, defines cell and reduced cell semimodules for them
and shows that under some assumptions these are exactly the 
maximal, with respect to projections, objects among minimal proper semimodules.
Finally, Section~\ref{s9} provides classification of extreme proper
semimodules for the Kazhdan-Lusztig semiring of a dihedral group.
\vspace{0.5cm}

{\bf Acknowledgment.}
This research was partially supported by the Swedish Research Council,
the G{\"o}ran Gustafsson Foundation and Vergstiftelsen.
\vspace{2cm}

\section{Basics}\label{s2}

\subsection{Semirings}\label{s2.1}

A {\em semiring} is a tuple $(R,+,\cdot,0,1)$ where
\begin{itemize}
\item $R$ is a set;
\item $+$ and $\cdot$ are binary operations on $R$;
\item $0$ and $1$ are elements in $R$;
\end{itemize}
which satisfies the following axioms:
\begin{itemize}
\item $(R,+,0)$ is a commutative monoid;
\item $(R,\cdot,1)$ is a monoid;
\item $(a+b)c=ac+bc$ and $c(a+b)=ca+cb$, for all $a,b,c\in R$;
\item $0a=a0=0$, for all $a\in R$.
\end{itemize}
Our basic example of a semiring is the semiring $(\mathbb{Z}_{\geq 0},+,\cdot,0,1)$ of 
non-negative integers with respect to the usual addition and multiplication.
For simplicity, in what follows we will refer to this semiring as $\mathbb{Z}_{\geq 0}$.

Let $\mathbf{Mon}$ denote the category of all monoids and monoid homomorphisms.
Another example of a semiring is the semiring $(\mathrm{End}_{\mathbf{Mon}}(M),+,\circ,\mathbf{0}_M,\mathrm{Id}_M)$, where
\begin{itemize}
\item $(M,+_M,0_M)$ is a {\em commutative} monoid;
\item $\mathrm{End}_{\mathbf{Mon}}(M)$ is the set of all endomorphisms of $M$ in $\mathbf{Mon}$;
\item $+$ is the usual addition of endomorphisms defined via
\begin{displaymath}
\big(\varphi+\psi\big)(m):=\varphi(m)+_M\psi(m),\quad \text{ for all }m\in M;
\end{displaymath}
\item $\circ$ is composition of endomorphisms;
\item $\mathbf{0}_M$ is the zero endomorphism of $M$, it is given by $\mathbf{0}_M(m)=0_M$, for all $m\in M$;
\item $\mathrm{id}_M$ is the identity endomorphism of $M$.
\end{itemize}
Again, for simplicity, we will refer to this semiring as $\mathrm{End}_{\mathbf{Mon}}(M)$.

Given two semirings $(R,+_R,\cdot_R,0_R,1_R)$ and $(T,+_T,\cdot_T,0_T,1_T)$, a {\em homomorphism}
of semirings is a map $\varphi:R\to T$ such that 
\begin{itemize}
\item $\varphi(a+_R b)=\varphi(a)+_T\varphi(b)$, for all $a,b\in R$;
\item $\varphi(a\cdot_R b)=\varphi(a)\cdot_T\varphi(b)$, for all $a,b\in R$;
\item $\varphi(0_R)=0_T$;
\item $\varphi(1_R)=1_T$.
\end{itemize}
For example, the identity map $\mathrm{id}_R$ is a homomorphism of semirings, for any semiring $R$.
Composition of homomorphisms of semirings is a homomorphism of semirings. Therefore
all semirings, together with all homomorphisms between semirings, form a category, denoted
$\mathbf{SRing}$.

For a semiring $R=(R,+,\cdot,0,1)$, the {\em opposite} semiring $R^{\mathrm{op}}$ is defined as the
semiring $(R,+,\cdot^{\mathrm{op}},0,1)$ where $a \cdot^{\mathrm{op}} b:=ba$, for all $a,b\in R$.

\subsection{Representations and semimodules}\label{s2.2}

Given a semiring $R=(R,+,\cdot,0,1)$, a {\em representation of $R$} is a semiring homomorphism
$\varphi:R\to \mathrm{End}_{\mathbf{Mon}}(M)$, for some commutative monoid $M=(M,+_M,0_M)$. The
monoid $M$ is called the {\em underlying monoid} of the representation $\varphi$. 

Given $R$ and $M$ as above, a {\em left $R$-semimodule} structure on $M$ is a map
\begin{displaymath}
R\times M\to M,\quad (r,m)\mapsto r(m), 
\end{displaymath}
satisfying the following axioms:
\begin{itemize}
\item $r(m+_M n)=r(m)+_Mr(n)$, for all $r\in R$ and $m,n\in M$;
\item $r(0_M)=0_M$, for all $r\in R$;
\item $(r+s)(m)=r(m)+_Ms(m)$, for all $r,s\in R$ and $m\in M$;
\item $(rs)(m)=r(s(m))$, for all $r,s\in R$ and $m\in M$;
\item $0(m)=0_M$, for all $m\in M$;
\item $1(m)=m$, for all $m\in M$.
\end{itemize}
There are several variations of the axioms for this structure in the literature, see \cite{Go,GM}.
An $R$-semimodule $M$ is called an {\em $R$-module} provided that $M$ is an abelian group.
An $R$-semimodule which is not a module will be called a {\em proper} semimodule.

Clearly, given a representation $\varphi:R\to \mathrm{End}_{\mathbf{Mon}}(M)$, the map
$(r,m)\mapsto \varphi(r)(m)$ defines a left $R$-semimodule structure on $M$. Conversely, any
left $R$-semimodule structure $(r,m)\mapsto r(m)$ gives rise to a representation of $R$
via $\varphi(r):=r({}_-)$. Therefore we will use the words {\em representation} and {\em (left) semimodule}
as synonyms, as usual. If not explicitly stated otherwise, by {\em semimodule} we always mean a {\em left semimodule}.

Here are some basic examples of $R$-semimodules.
\begin{itemize}
\item If $M=\{0_M\}$, then the unique map $R\to \mathrm{End}_{\mathbf{Mon}}(M)$ is 
a representation of $R$ called the {\em zero semimodule}.
\item The commutative monoid $(R,+,0)$ has the natural structure of a left $R$-semimodule
via $(r,m)\mapsto rm$, for $r,m\in R$. This is the {\em left regular} $R$-semimodule, usually denoted
${}_RR$.
\end{itemize}

Given two $R$-semimodules $M=(M,+_M,0_M)$ and $N=(N,+_N,0_N)$, an {\em $R$-semimodule homomorphism} from
$M$ to $N$ is a map $\alpha:M\to N$ such that 
\begin{itemize}
\item $\alpha$ is a homomorphism of monoids;
\item $\alpha(r(m))=r(\alpha(m))$, for all $r\in R$ and $m\in M$, that is, the following diagram
commutes:
\begin{displaymath}
\xymatrix{ 
M\ar[rr]^{r({}_-)}\ar[d]_{\alpha}&&M\ar[d]^{\alpha}\\
N\ar[rr]^{r({}_-)}&&N
}
\end{displaymath}
\end{itemize}
Here are some examples of $R$-semimodule homomorphisms:
\begin{itemize}
\item For any $R$-semimodule $M$, the identity map $\mathrm{id}_M:M\to M$ is an $R$-semimodule homomorphism.
\item For any $R$-semimodules $M$ and $N$, the zero map $\mathbf{0}_{M,N}:M\to N$, given by
$\mathbf{0}_{M,N}(m)=0_N$, for all $m\in M$, is an $R$-semimodule homomorphism.
\end{itemize}
Composition of  $R$-semimodule homomorphisms is an $R$-semimodule homomorphism. Therefore all
$R$-semimodules and  $R$-semimodule homomorphisms form a category, denoted $R\text{-}\mathrm{sMod}$.
The set of all $R$-semimodule homomorphisms from $M$ to $N$ is denoted 
\begin{displaymath}
\mathrm{Hom}_R(M,N)=R\text{-}\mathrm{sMod}(M,N).
\end{displaymath}

The notion of a {\em right} $R$-semimodule is defined mutatis mutandis to the above. The
{\em right regular} $R$-semimodule is denoted $R_R$. The 
category of right $R$-semimodules is denoted $\mathrm{sMod}\text{-}R$. As usual, we have a canonical
isomorphism of categories 
\begin{displaymath}
\mathrm{sMod}\text{-}R\cong R^{\mathrm{op}}\text{-}\mathrm{sMod}. 
\end{displaymath}

\subsection{Bisemimodules}\label{s2.3}

An {\em $R$-$R$-bisemimodule} is a commutative monoid $M$ equipped both with the
structure of a left $R$-semimodule $(a,m)\mapsto am$ and with the
structure of a right $R$-semimodule $(m,b)\mapsto mb$ such that these two
structures commute in the sense that $a(mb)=(am)b$, for all $a,b\in R$ and $m\in M$.
The semiring $R$ itself has the natural structure of an $R$-$R$-bisemimodule given by
multiplication. This bisemimodule is called the {\em regular} $R$-$R$-bisemimodule and denoted ${}_RR_R$. 
The  category of $R$-$R$-bisemimodules is denoted $R\text{-}\mathrm{sMod}\text{-}R$.

\subsection{Subsemimodules}\label{s2.4}

Let $R$ be a semiring and $M\in R\text{-}\mathrm{sMod}$. A {\em subsemimodule} of $M$ is a submonoid
$N$ of $M$ which is closed under the action of $R$. For example, both ${0_M}$ and $M$ are always
subsemimodules of $M$. A subsemimodule of $M$ different from ${0_M}$ and $M$ is called a {\em proper} subsemimodule. Similar terminology is also used for right semimodules.

Subsemimodules of regular semimodules have special names:
\begin{itemize}
\item Subsemimodules of ${}_RR$ are called {\em left ideals} of $R$.
\item Subsemimodules of $R_R$ are called {\em right ideals} of $R$.
\item Subbisemimodules of ${}_RR_R$ are called {\em ideals} of $R$.
\end{itemize}

\subsection{Congruences and quotients}\label{s2.5}

Let $R$ be a semiring and $M\in R\text{-}\mathrm{sMod}$. An equivalence relation $\sim$ on $M$
is called an {\em $R$-congruence} provided that
\begin{itemize}
\item $m\sim n$ implies $(m+k)\sim(n+k)$, for all $m,n,k\in M$;
\item $m\sim n$ implies $r(m)\sim r(n)$, for all $m,n\in R$ and $r\in R$.
\end{itemize}
The first of these two conditions means that $\sim$ is a congruence on the monoid $M$.

Given an $R$-congruence $\sim$ on $M$, the set $M/_\sim$ of equivalence classes with respect to 
$\sim$ has the natural structure of an $R$-semimodule given by
\begin{itemize}
\item $(m)_{\sim}+(n)_{\sim}:=(m+n)_{\sim}$, for all $m,n\in M$;
\item the zero element of $M/_\sim$ is $(0_M)_{\sim}$;
\item $r\big((m)_{\sim}\big):=(r(m))_{\sim}$, for all $r\in R$ and $m\in M$.
\end{itemize}
Here we denote by $(m)_{\sim}$ the $\sim$-equivalence class containing an element $m\in M$.
The semimodule $M/_\sim$ is called the {\em quotient} of $M$ with respect to $\sim$.

Let $N\in R\text{-}\mathrm{sMod}$ and $\alpha\in\mathrm{Hom}_R(M,N)$. Define 
an equivalence relation $\sim_{\alpha}$ on $M$ via $m\sim_{\alpha}k$ if and only if 
$\alpha(m)=\alpha(k)$, for $m,k\in M$. The equivalence relation $\sim_{\alpha}$ is called
the {\em kernel} of $\alpha$. 

\begin{proposition}\label{prop2.1}
The equivalence relation $\sim_{\alpha}$ is an $R$-congruence and $M/_{\sim_{\alpha}}$
is isomorphic to the image $\mathrm{Im}(\alpha)$ of $\alpha$, the latter being a subsemimodule of $N$.
\end{proposition}

\begin{proof}
The map $(m)_{\sim_{\alpha}}\mapsto \alpha(m)$ defines an isomorphism 
between $M/_{\sim_{\alpha}}$ and $\mathrm{Im}(\alpha)$.
\end{proof}

Assume that the monoid $M$ is, in fact, a group and $\sim$ an ($R$-)congruence on $M$.
Then there exists a subgroup $N$ of $M$ such that equivalence classes
of $\sim$ are exactly the cosets in $M/N$, see e.g. \cite[Subsection~6.2]{GaMa}.

\subsection{Minimal, elementary and simple semimodules}\label{s2.6}

A non-zero $R$-semimodule $M$ is called
\begin{itemize}
\item {\em minimal} provided that the only subsemimodules of $M$ are $\{0_M\}$ and $M$;
\item {\em elementary} provided that the only $R$-congruences on $M$ are 
the {\em equality relation} $=_M$ and the {\em full relation} $M\times M$;
\item {\em simple} provided that it is both minimal and elementary. 
\end{itemize}
Obviously, any $R$-semimodule $M$ such that $|M|=2$ is simple.

As usual, we will use {\em simple semimodules} and {\em irreducible representations} as synonyms.
Note the difference in terminology with, in particular, \cite{IRS}.

The following property of minimal semimodules is noted in \cite[Proposition~2.7]{IRS}.

\begin{proposition}\label{prop2.2}
Any non-zero quotient of a minimal $R$-semimodule is minimal.
\end{proposition}

\begin{proof}
This follows directly from the observation that the full preimage of a subsemimodule
is a subsemimodule.
\end{proof}

\begin{example}\label{ex2.25}
{\rm
Let $R=(\mathbb{R}_{\geq 0},+,\cdot,0,1)$. Then the left regular $R$-semimodule ${}_RR$
is minimal as $\mathbb{R}_{\geq 0}a=\mathbb{R}_{\geq 0}$, for any 
$a\in\mathbb{R}_{>0}$. At the same time, the equivalence relation
$\sim$ on ${}_RR$ with equivalence classes $\{0\}$ and
$\mathbb{R}_{>0}$ is a non-trivial $R$-congruence. Hence this semimodule is not elementary.
} 
\end{example}

\begin{example}\label{ex2.26}
{\rm
Let $R=(\mathbb{R}_{\geq 0},+,\cdot,0,1)$ and $M=(\mathbb{R},+,0)$. Then 
$R$ acts on $M$ via multiplication. The set of all non-negative elements of $M$
is an $R$-invariant submonoid. Hence the $R$-semimodule $M$ is not minimal.
We claim that $M$ is elementary. As $M$ is a group, any congruence on $M$
has the form of cosets with respect to some subgroup, see \cite[Subsection~6.2]{GaMa}. 
We claim that $M$ has no proper $R$-invariant subgroups. Indeed, if an $R$-invariant
subgroup of $M$ contains a non-zero element $m$, then $Rm$ is the set of all
elements of $M$ having the same sign as $m$. The subgroup of $M$ generated by such $Rm$
equals $M$. This example overlaps with \cite[Example~3.7(c)]{KNZ}.
}
\end{example}

In this subsection (and in the rest of this paper) 
we slightly deflect from the terminology used in \cite{Go,GM}.
We find our terminology better adjusted to the fact that simple modules over rings are
both minimal and elementary in the above sense.

We will also say that a semimodule $M$ is {\em extreme} provided that it is 
minimal {\em or} elementary (or both).

\subsection{Direct sums of semimodules}\label{s2.7}

Given $M,N\in R\text{-}\mathrm{sMod}$, we have their {\em direct sum} $M\oplus N\in R\text{-}\mathrm{sMod}$
defined in the usual way as the set of all pairs $(m,n)$, where $m\in M$ and $n\in N$, with component-wise
operations. We have the usual inclusion $R$-homomorphisms $\iota_M:M\to M\oplus N$, given by
$m\mapsto (m,0_N)$, for all $m\in M$, and $\iota_N:N\to M\oplus N$, given by $n\mapsto (0_M,n)$, for $n\in N$. 
We also have the usual projection $R$-homomorphisms $\pi_M:M\oplus N\to M$, given by
$(m,n)\mapsto m$, for all $(m,n)\in M\oplus N$, and $\pi_N:M\oplus N\to N$, given by $(m,n)\mapsto n$, for all 
$(m,n)\in M\oplus N$. 

In general, this is not sufficient for $R\text{-}\mathrm{sMod}$ to be an additive category since $R\text{-}\mathrm{sMod}$
is, usually, not even preadditve in the sense that morphism spaces in $R\text{-}\mathrm{sMod}$ are usually
not abelian groups (they are only abelian monoids).

As usual, we write $M^{\oplus k}$ for $M\oplus M\oplus\dots\oplus M$ ($k$ summands).

\subsection{Free semimodules}\label{s2.8}

Let $R$ be a semi-ring, $M\in R\text{-}\mathrm{sMod}$ and $B$ a non-empty subset of $M$.
As usual, we will say that $M$ is {\em free with basis} $B$ provided that, for any 
$N\in R\text{-}\mathrm{sMod}$ and any map $f:B\to N$, there is a unique 
$\alpha\in\mathrm{Hom}_R(M,N)$ such that $\alpha\vert_{B}=f$. In other words, the following
diagram commutes.
\begin{displaymath}
\xymatrix{
B\ar@{^{(}->}[rr]\ar[rrd]_{f}&& M\ar@{.>}[d]^{\exists ! \alpha}\\
&&N
}
\end{displaymath}

For any positive integer $k$, the $R$-semimodule $R^{\oplus k}$ is free with a basis given
by standard basis vectors $e_i$, where $i=1,2,\dots,k$. For a fixed cardinality of a basis,
free semimodules (if they exist) are unique up to isomorphism, as follows directly from the universal
property above.

An $R$-semimodule $M$ is called {\em finitely generated} provided that it is isomorphic to a quotient of 
the semimodule $R^{\oplus k}$, for some positive integer $k$. The category of all finitely generated
$R$-semimodules is denoted $R\text{-}\mathrm{smod}$.

\section{Various Boolean semimodules}\label{s25}

\subsection{Boolean semiring}\label{s25.1}

We denote by $\mathbf{B}$ the {\em Boolean semiring} 
\begin{displaymath}
\mathbf{B}=(\{0,1\},+,\cdot,0,1), 
\end{displaymath}
where $+$ and $\cdot $ are given by their Cayley tables
\begin{displaymath}
\begin{array}{c||c|c}
+& 0&1\\
\hline\hline
0&0&1\\
\hline
1&1&1
\end{array}\qquad\qquad
\begin{array}{c||c|c}
\cdot& 0&1\\
\hline\hline
0&0&0\\
\hline
1&0&1
\end{array},
\end{displaymath}
respectively. 

\subsection{$\mathbf{B}$-semimodules}\label{s25.2}

Recall that a semigroup consisting of idempotents is called a {\em band}.
Abelian bands are exactly the {\em semi-lattices}.

\begin{lemma}\label{lem25.1}
The following statements hold.
\begin{enumerate}[$($i$)$]
\item\label{lem25.1.1} Every $B$-semimodule is a semilattice. 
\item\label{lem25.1.2} Every monoid semilattice has the unique structure of a $B$-semimodule, namely the one where
$0$ acts as zero and $1$ acts as the identity.
\end{enumerate}
\end{lemma}

\begin{proof}
Let $M\in \mathbf{B}\text{-}\mathrm{sMod}$. Then $1+1=1$ implies $m+m=m$, for any $m\in M$.
This proves Claim~\eqref{lem25.1.1}. Conversely, if $M$ is a 
semilattice, then $\mathrm{id}_M+\mathrm{id}_M=\mathrm{id}_M$
which implies Claim~\eqref{lem25.1.2}.
\end{proof}

Let $\mathbf{SLat}^1$ denote the category of all monoid semi-lattices. From Lemma~\ref{lem25.1} it 
follows that the category $\mathbf{B}\text{-}\mathrm{sMod}$ is isomorphic to $\mathbf{SLat}^1$.

\subsection{Extreme $\mathbf{B}$-semimodules}\label{s25.3}

\begin{theorem}\label{thm25.2}
Let $M\in \mathbf{B}\text{-}\mathrm{sMod}$. Then the following conditions are equivalent.
\begin{enumerate}[$($i$)$]
\item\label{thm25.2.1} $M$ is minimal.
\item\label{thm25.2.2} $M$ is elementary. 
\item\label{thm25.2.3} $M$ is simple. 
\item\label{thm25.2.4} $M$ is isomorphic to the left regular semimodule ${}_{\mathbf{B}}\mathbf{B}$. 
\end{enumerate}
\end{theorem}

\begin{proof}
As $|\mathbf{B}|=2$, Claim~\eqref{thm25.2.4} implies all other claims.

On the other hand, let $M\in \mathbf{B}\text{-}\mathrm{sMod}$ with identity $0$. Then, for any $m\in M$, 
from Lemma~\ref{lem25.1} it follows that $\{0,m\}$ is a subsemimodule. Therefore minimality 
of $M$ implies $M=\{0,m\}$. As any semimodule with two elements is simple, we get \eqref{thm25.2.1}$\Rightarrow$\eqref{thm25.2.3}$\Rightarrow$\eqref{thm25.2.4}.

Finally, let $M\in \mathbf{B}\text{-}\mathrm{sMod}$ with identity $0$  be elementary. Then all
$m\in M$ different from $0$ form a $\mathbf{B}$-stable ideal in $M$. As $M$ is elementary, it follows
that this ideal must contain exactly one element. This gives the implication 
\eqref{thm25.2.2}$\Rightarrow$\eqref{thm25.2.4} and completes the proof.
\end{proof}

\subsection{Simple boolean representations of finite groups}\label{s25.4}

Let $G$ be a finite group and $\mathbf{B}[G]$ its group semiring over $\mathbf{B}$.
We can consider $\mathbf{B}$ as the {\em trivial} $\mathbf{B}[G]$-semimodule where each
$g\in G$ acts as the identity.

\begin{theorem}\label{thm25.3}
Let $M\in \mathbf{B}[G]\text{-}\mathrm{sMod}$. Then the following conditions are equivalent.
\begin{enumerate}[$($i$)$]
\item\label{thm25.3.1} $M$ is minimal.
\item\label{thm25.3.2} $M$ is elementary. 
\item\label{thm25.3.3} $M$ is simple. 
\item\label{thm25.3.4} $M$ is isomorphic to the trivial semimodule. 
\end{enumerate}
\end{theorem}

A substantial part of this theorem is contained in \cite[Theorem~4.4]{IRS}.

\begin{proof}
The implications \eqref{thm25.3.3}$\Rightarrow$\eqref{thm25.3.1} and
\eqref{thm25.3.3}$\Rightarrow$\eqref{thm25.3.2} follow from the definitions. The trivial $\mathbf{B}[G]$-semimodule 
has only two elements and hence is simple. This gives the implication
\eqref{thm25.3.4}$\Rightarrow$\eqref{thm25.3.3}.

Let $M\in \mathbf{B}[G]\text{-}\mathrm{sMod}$ be minimal and $m\in M$ be a non-zero element.
As each $g\in G$ acts on $M$ via an automorphism, we have $g(m)\neq 0$, for all $g\in G$.
As non-zero elements of a monoidal commutative band form an ideal, it follows that the element
$n:=\sum_{g\in G}g(m)$ is non-zero. This gives that $\{0,n\}$ is a non-zero $\mathbf{B}[G]$-sub\-semi\-module.
By minimality, we have $\{0,n\}=M$,
establishing the implication \eqref{thm25.3.1}$\Rightarrow$\eqref{thm25.3.4}.
 
Let $M\in \mathbf{B}[G]\text{-}\mathrm{sMod}$ be elementary. As mentioned in the previous
paragraph, the set of all non-zero elements of  $M$ forms $\mathbf{B}[G]$-stable ideal.
As $M$ is elementary, the corresponding Rees congruence must be the equality relation. Therefore
$M$ has only two elements and thus is isomorphic to the trivial $\mathbf{B}[G]$-subsemimodule.
This gives the implication \eqref{thm25.3.2}$\Rightarrow$\eqref{thm25.3.4}
and completes the proof.
\end{proof}

\subsection{Extreme semimodules over finite cardinality semirings}\label{s25.5}

For a non-negative integer $k$, we denote by $\mathbb{N}_k$ the Rees quotient of the semiring
$(\mathbb{Z}_{\geq 0},+,\cdot,0,1)$ modulo the congruence with classes  
$\{0\}, \{1\},\dots,\{k-1\},I_k:=\{m\,:\,m\geq k\}$. We have 
$\mathbf{B}\cong \mathbb{N}_1$. The semirings $\mathbb{N}_k$ appear naturally in the theory of
multisemigroups with multiplicities developed in \cite{Fo}.

The map $\psi_k:\mathbb{N}_k\to \mathbf{B}$, given by 
\begin{displaymath}
\psi_k(i):=
\begin{cases}
0, & i=0,\\
1, & i\neq 0;
\end{cases}
\end{displaymath}
is a homomorphism of semirings. Via $\psi$, we may view the left regular $\mathbf{B}$-semimodule
${}_\mathbf{B}\mathbf{B}$ as an $\mathbb{N}_k$-semimodule, denoted ${}_{\mathbb{N}_k}\mathbf{B}$.

\begin{theorem}\label{thm25.4}
Let $M\in \mathbb{N}_k\text{-}\mathrm{sMod}$. Then the following conditions are equivalent.
\begin{enumerate}[$($i$)$]
\item\label{thm25.4.1} $M$ is minimal.
\item\label{thm25.4.2} $M$ is elementary. 
\item\label{thm25.4.3} $M$ is simple. 
\item\label{thm25.4.4} $M$ is isomorphic to ${}_{\mathbb{N}_k}\mathbf{B}$. 
\end{enumerate}
\end{theorem}

\begin{proof}
As $|\mathbf{B}|=2$, Claim~\eqref{thm25.4.4} implies all other claims, in particular, we have
\eqref{thm25.4.4}$\Rightarrow$\eqref{thm25.4.3}$\Rightarrow$\eqref{thm25.4.1} and
\eqref{thm25.4.4}$\Rightarrow$\eqref{thm25.4.3}$\Rightarrow$\eqref{thm25.4.2}.

Let now $M\in \mathbb{N}_k\text{-}\mathrm{sMod}$ be minimal and $0\neq m\in M$. Then $i\cdot m=k\cdot m$,
for all $i\geq k$. In particular, $k\cdot m\neq 0$. This implies that $\{0,k\cdot m\}$ is an
$\mathbb{N}_k$-subsemimodule of $M$. Consequently, $\{0,k\cdot m\}=M$ due to the minimality of $M$.
Therefore \eqref{thm25.4.1}$\Rightarrow$\eqref{thm25.4.4}.

Finally, let $M\in \mathbb{N}_k\text{-}\mathrm{sMod}$ be elementary. The argument of the previous paragraph
implies that $0$ is the only invertible element of $M$. Therefore $M$ must contain non-invertible elements.
As $M$ is commutative, all non-invertible elements form an ideal. This ideal is, clearly, invariant under the
action of $\mathbb{N}_k$. The Rees quotient modulo this ideal is, clearly, isomorphic to ${}_{\mathbb{N}_k}\mathbf{B}$.
Hence we have $M\cong {}_{\mathbb{N}_k}\mathbf{B}$ since $M$ is elementary.
Therefore \eqref{thm25.4.2}$\Rightarrow$\eqref{thm25.4.4} and the proof is complete.
\end{proof}

For any finite group $G$, we can consider ${}_{\mathbb{N}_k}\mathbf{B}$ as the {\em trivial}
$\mathbb{N}_k[G]$-semimodule where all $g\in G$ act as the identity. This semimodule is, clearly, simple.
\vspace{1cm}

\begin{theorem}\label{thm25.5}
Let $M\in \mathbb{N}_k[G]\text{-}\mathrm{sMod}$. Then the following conditions are equivalent.
\begin{enumerate}[$($i$)$]
\item\label{thm25.5.1} $M$ is minimal.
\item\label{thm25.5.2} $M$ is elementary. 
\item\label{thm25.5.3} $M$ is simple. 
\item\label{thm25.5.4} $M$ is isomorphic to the trivial semimodule. 
\end{enumerate}
\end{theorem}

\begin{proof}
We, clearly,  have \eqref{thm25.5.4}$\Rightarrow$\eqref{thm25.5.3}$\Rightarrow$\eqref{thm25.5.1} and
\eqref{thm25.5.4}$\Rightarrow$\eqref{thm25.5.3}$\Rightarrow$\eqref{thm25.5.2}. 
 
To prove \eqref{thm25.5.1}$\Rightarrow$\eqref{thm25.5.4}, let $M$ be a minimal  
$\mathbb{N}_k[G]$-semimodule and $0\neq m\in M$. Then $i\cdot m=k\cdot m$,
for all $i\geq k$. In particular, $k\cdot m\neq 0$ is an (additive) idempotent. Let 
$n:=\sum_{g\in G}g(k\cdot m)=\sum_{g\in G}(k\cdot g(m))$. Then $n$ is an additive idempotent and $\{0,n\}$
is a $G$-invariant $\mathbb{N}_k$-subsemimodule of $M$. Hence $M=\{0,n\}$ implying
Claim~\eqref{thm25.5.4}.

To prove \eqref{thm25.5.2}$\Rightarrow$\eqref{thm25.5.4}, let $M$ be an elementary  
$\mathbb{N}_k[G]$-semimodule and  $0\neq m\in M$. Then $i\cdot m=k\cdot m$,
for all $i\geq k$. In particular, $k\cdot m\neq 0$ is an (additive) idempotent and thus
$m$ is not additively invertible. In particular, $M$ contains non-invertible elements. 
Hence the equivalence relation on $M$ with two equivalence classes, the set of all invertible
and the set of all non-invertible elements, is an $\mathbb{N}_k[G]$-congruence. As $M$ is elementary,
it follows that both equivalence classes of this congruence must be singletons, implying 
Claim~\eqref{thm25.5.4}.
\end{proof}

\section{Extreme $\mathbb{Z}_{\geq 0}$-semimodules}\label{s3}

\subsection{$\mathbb{Z}_{\geq 0}$-semimodules}\label{s3.1}

Let $\mathbf{AMon}$ denote the category of all commutative mo\-no\-ids and monoid homomorphisms.

\begin{proposition}\label{prop3.0}
Each commutative monoid $M$ has the unique structure of a $\mathbb{Z}_{\geq 0}$-se\-mi\-module, namely, the one 
given by
\begin{displaymath}
i(m)=
\begin{cases}
0_M,& i=0;\\
\underbrace{m+m+\dots+m}_{i\text{ times}},,& i>0; 
\end{cases}
\end{displaymath}
for $i\in \mathbb{Z}_{\geq 0}$ and $m\in M$.
Consequently, the categories $\mathbb{Z}_{\geq 0}\text{-}\mathrm{sMod}$ and 
$\mathbf{AMon}$ are canonically isomorphic.
\end{proposition}

\begin{proof}
This follows directly from the definitions. 
\end{proof}

\subsection{Some simple $\mathbb{Z}_{\geq 0}$-semimodules}\label{s3.2}

The map $\varphi:\mathbb{Z}_{\geq 0}\to \mathbf{B}$, defined by
\begin{displaymath}
\varphi(i)=
\begin{cases}
0,& i=0;\\
1, & i>0;
\end{cases}
\end{displaymath}
is a homomorphism of semirings. This defines on $\mathbf{B}$ the structure of a $\mathbb{Z}_{\geq 0}$-semimodule
via $i(m)=\varphi(i)m$, for $i\in \mathbb{Z}_{\geq 0}$ and $m\in \mathbf{B}$. Note that 
$|\mathbf{B}|=2$ and hence this $\mathbb{Z}_{\geq 0}$-semimodule is simple.

For a positive integer $n$, let $\mathbb{Z}_n$, as usual, denote the ring of integer residue classes modulo $n$.
The map 
\begin{displaymath}
i\mapsto i\,\mathrm{mod}\,n
\end{displaymath}
is a semiring homomorphism from $\mathbb{Z}_{\geq 0}$ to $\mathbb{Z}_n$.
Similarly to the above, this defines on $\mathbb{Z}_n$ the structure of a $\mathbb{Z}_{\geq 0}$-semimodule.

\begin{lemma}\label{lem3.1}
For $n$ as above, the following conditions are equivalent.
\begin{enumerate}[$($i$)$]
\item\label{lem3.1.1} The  $\mathbb{Z}_{\geq 0}$-semimodule $\mathbb{Z}_n$ is minimal.
\item\label{lem3.1.2} The  $\mathbb{Z}_{\geq 0}$-semimodule $\mathbb{Z}_n$ is elementary. 
\item\label{lem3.1.3} The  $\mathbb{Z}_{\geq 0}$-semimodule $\mathbb{Z}_n$ is simple. 
\item\label{lem3.1.4} The number $n$ is a prime number. 
\end{enumerate}
\end{lemma}

\begin{proof}
Note that $\mathbb{Z}_n$ is a finite group and hence a finite submonoid of 
$\mathbb{Z}_n$ is a subgroup. 
Any congruence on a group is given by cosets with respect to a subgroup, see \cite[Subsection~6.2]{GaMa}.
This means that the $\mathbb{Z}_{\geq 0}$-semimodule $\mathbb{Z}_n$ is simple if and only if
it is minimal if and only if it is elementary if and only if $\mathbb{Z}_n$ is a simple
group. The latter is the case if and only if  $n$ is prime. The claim follows.
\end{proof}

\subsection{Classification of extreme $\mathbb{Z}_{\geq 0}$-semimodules}\label{s3.3}

\begin{theorem}\label{thm3.2}
Let $M\in \mathbb{Z}_{\geq 0}\text{-}\mathrm{sMod}$. Then the following conditions are equivalent.
\begin{enumerate}[$($i$)$]
\item\label{thm3.2.1} $M$ is minimal.
\item\label{thm3.2.2} $M$ is elementary. 
\item\label{thm3.2.3} $M$ is simple. 
\item\label{thm3.2.4} $M$ is isomorphic to either $\mathbf{B}$ or 
$\mathbb{Z}_p$, for some prime $p$. 
\end{enumerate}
\end{theorem}

\begin{proof}
That Claim~\eqref{thm3.2.4} implies Claims~\eqref{thm3.2.1}, \eqref{thm3.2.2} and \eqref{thm3.2.3}
is shown in Subsection~\ref{s3.2}.

Assume first that $M$ is a group. Any congruence on a group is given by cosets with respect to a subgroup, 
see \cite[Subsection~6.2]{GaMa}. This implies that 
Claims~\eqref{thm3.2.1}, \eqref{thm3.2.2} and \eqref{thm3.2.3} for $M$ are equivalent,
moreover, they are also equivalent to the requirement that $M$ is a simple group. 
This means that $M\cong \mathbb{Z}_p$, for some prime $p$, and hence implies Claim~\eqref{thm3.2.4}.

Now let us assume that $M$ is not a group. As $M$ is commutative, all Green's relations on $M$ coincide.
As $M$ is not a group, it must have at least two different $\mathcal{J}$-classes. Let $I$ be the
ideal of $M$ consisting of all non-invertible elements. Note that $I\neq M$. The map $\varphi:M\to \mathbf{B}$
given by
\begin{displaymath}
\varphi(m)=
\begin{cases}
0, & m\not\in I;\\
1, & m\in I; 
\end{cases}
\end{displaymath}
is an $\mathbb{Z}_{\geq 0}$-homomorphism. Therefore Claims~\eqref{thm3.2.2} and \eqref{thm3.2.3} imply
$M\cong \mathbf{B}$ which gives Claim~\eqref{thm3.2.4}.

Finally, assume Claim~\eqref{thm3.2.1}. Note that $\{0_M\}\cup I$ is a subsemimodule of $M$.
This implies that $M\setminus I=\{0_M\}$. Take any $m\in I$. If all elements $im$, where $i\in\mathbb{Z}_{\geq 0}$,
were different, then $\{0_M,2m,3m,4m,\dots\}$ would be a proper subsemimodule of $M$, which is not possible.
Hence the set $\mathbb{Z}_{\geq 0}m$ is finite. Clearly, $\{0_M\}\cup \mathbb{Z}_{\geq 0}m$ is a subsemimodule
of $M$ and hence coincides with $M$ due to Claim~\eqref{thm3.2.1}. Further, 
$\mathbb{Z}_{\geq 0}m$ is a finite semigroup and hence contains an idempotent, say $x$.
Then $\{0_M,x\}$ is a  subsemimodule
of $M$ and hence coincides with $M$ by Claim~\eqref{thm3.2.1}. The map
$0_M\mapsto 0$ and $x\mapsto 1$ is an isomorphism from $M$ to $\mathbf{B}$ which again gives Claim~\eqref{thm3.2.4}.
The proof is complete.
\end{proof}

\subsection{$\mathbb{Z}_{\geq 0}$-modules}\label{s3.4}

The natural embedding of the semiring $\mathbb{Z}_{\geq 0}$ into the ring $\mathbb{Z}$ is a homomorphism
of unital semirings. Therefore we have the corresponding restriction functor
\begin{displaymath}
\mathrm{Res}^{\mathbb{Z}}_{\mathbb{Z}_{\geq 0}}:\mathbb{Z}\text{-}\mathrm{sMod}\longrightarrow
\mathbb{Z}_{\geq 0}\text{-}\mathrm{sMod}.
\end{displaymath}
Note that $\mathbb{Z}\text{-}\mathrm{sMod}=\mathbb{Z}\text{-}\mathrm{Mod}$. 

\begin{proposition}\label{prop3.3}
The restriction induces an isomorphism of categories 
\begin{displaymath}
\mathrm{Res}^{\mathbb{Z}}_{\mathbb{Z}_{\geq 0}}:\mathbb{Z}\text{-}\mathrm{Mod}\longrightarrow
\mathbb{Z}_{\geq 0}\text{-}\mathrm{Mod}.
\end{displaymath}
\end{proposition}

\begin{proof}
Each  $\mathbb{Z}_{\geq 0}$-module is just an abelian group and hence a 
$\mathbb{Z}$-module in a unique way. If $M$ and $N$ are abelian groups and $\varphi:M\to N$
is a monoid homomorphism, then it is a group homomorphism. The claim follows.
\end{proof}

\section{Extreme $\mathbb{Z}_{\geq 0}[S_2]$-semimodules}\label{s5}

\subsection{$\mathbb{Z}_{\geq 0}[S_2]$-semimodules}\label{s5.1}

Let $S_2=\{e,s=(12)\}$ be the symmetric group on $\{1,2\}$. Let $\mathbb{Z}_{\geq 0}[S_2]$ be the
group semiring with coefficients in $\mathbb{Z}_{\geq 0}$.

From Proposition~\ref{prop3.0} it follows that a $\mathbb{Z}_{\geq 0}[S_2]$-semimodule can be understood
as a pair $(M,\tau)$, where $M$ is a commutative monoid and $\tau:M\to M$ is an involutive automorphism.
The automorphism $\tau$ represents the action of $s$.
A homomorphism $\varphi:(M,\tau)\to (N,\sigma)$ of $\mathbb{Z}_{\geq 0}[S_2]$-semimodules is a homomorphism of the underlying monoids which
intertwines $\tau$ and $\sigma$. A subsemimodule of $(M,\tau)$ is just a $\tau$-stable submonoid of $M$.
A $\mathbb{Z}_{\geq 0}[S_2]$-congruence on $(M,\tau)$ is a $\tau$-compatible congruence on $M$.

Every commutative monoid $M$ has the {\em trivial} structure of a $\mathbb{Z}_{\geq 0}[S_2]$-semimodule
given by $\tau=\mathrm{id}_M$. Such $M$ will be called {\em trivial extensions} of 
$\mathbb{Z}_{\geq 0}$-semimodules.

For a prime $p>2$, let $\tau_p:\mathbb{Z}_p\to \mathbb{Z}_p$ denote the automorphism
which sends $1$ to $p-1$. The $\mathbb{Z}_{\geq 0}[S_2]$-semimodule $(\mathbb{Z}_p,\tau_p)$ 
is, clearly,  simple.

\subsection{Extreme $\mathbb{Z}_{\geq 0}[S_2]$-semimodules}\label{s5.2}

\begin{theorem}\label{thm5.1}
Let $M\in \mathbb{Z}_{\geq 0}[S_2]\text{-}\mathrm{sMod}$. Then the following conditions are equivalent.
\begin{enumerate}[$($i$)$]
\item\label{thm5.1.1} $M$ is minimal.
\item\label{thm5.1.2} $M$ is elementary. 
\item\label{thm5.1.3} $M$ is simple. 
\item\label{thm5.1.4} $M$ is isomorphic to either the trivial extension of a 
simple $\mathbb{Z}_{\geq 0}$-semimodule or to $(\mathbb{Z}_p,\tau_p)$, for some prime $p>2$.
\end{enumerate}
\end{theorem}

\begin{proof}
Clearly, \eqref{thm5.1.4}$\Rightarrow$\eqref{thm5.1.3}$\Rightarrow$\eqref{thm5.1.1} and
\eqref{thm5.1.4}$\Rightarrow$\eqref{thm5.1.3}$\Rightarrow$\eqref{thm5.1.2}.
 
Let $M$ be a minimal $\mathbb{Z}_{\geq 0}[S_2]$-semimodule. Assume that $M$ has a non-invertible
element, say $m$. Then $s(m)$ is non-invertible and so is $m+s(m)$. The latter element is $s$-invariant.
Therefore $s$ acts as the identity on the submodule of $M$ generated by $m+s(m)$. This means that 
$M$ is isomorphic to the trivial extension of a simple $\mathbb{Z}_{\geq 0}$-semimodule.

Assume now that $M$ is a $\mathbb{Z}_{\geq 0}[S_2]$-module and that the action of $s$ is not trivial.
For any non-zero $m\in M$, we have $m+s(m)$ is $s$-invariant and hence generates a submodule of 
$M$ with trivial $s$-action. As $M$ is minimal, we thus have $m+s(m)=0$, for all $m\in M$.
In other words, $s(m)=-m$. Therefore any subgroup of $M$ is automatically $s$-invariant.
Hence minimality of $M$ implies $M\cong \mathbb{Z}_p$, for some prime $p$. If $p=2$, then the action
of $s$ is trivial. Hence $p>2$. This proves the implication \eqref{thm5.1.1}$\Rightarrow$\eqref{thm5.1.4}.

Let $M$ be an elementary $\mathbb{Z}_{\geq 0}[S_2]$-semimodule. If $M$ is a $\mathbb{Z}_{\geq 0}[S_2]$-module,
then, for any $m\in M$, the element $m+s(m)$ generates an $s$-invariant subgroup of $M$, say $N$.
As $M$ is elementary, we get that either $N=M$ or $N=\{0\}$. In the first case, $s$ acts on $M$
as the identity and hence $M$ is a simple $\mathbb{Z}_{\geq 0}$-module. In the second case 
$s$ acts on $M$ via the negation, in particular, any subgroup is $s$-invariant.
Altogether, we have that $M\cong \mathbb{Z}_p$ with $s$ acting either
trivially or via $\tau_p$, for $p>2$.

If $M$ has a non-invertible element, then we have a 
$\mathbb{Z}_{\geq 0}[S_2]$-congruence $\sim$ on $M$ with two equivalence classes given by all invertible 
and all non-invertible elements, respectively. Therefore the fact that $M$ is elementary implies that 
$M$ has two elements and the action of $s$ on $M$ is trivial. This proves  
the implication \eqref{thm5.1.2}$\Rightarrow$\eqref{thm5.1.4} and completes the proof.
\end{proof}

\subsection{Integral vs non-negative integral scalars}\label{s5.3}

The natural embedding of the semiring $\mathbb{Z}_{\geq 0}[S_2]$ into the ring 
$\mathbb{Z}[S_2]$ is a homomorphism of unital semirings. Therefore we have the corresponding restriction functor
\begin{displaymath}
\mathrm{Res}^{\mathbb{Z}[S_2]}_{\mathbb{Z}_{\geq 0}[S_2]}:\mathbb{Z}[S_2]\text{-}\mathrm{sMod}\longrightarrow
\mathbb{Z}_{\geq 0}[S_2]\text{-}\mathrm{sMod}.
\end{displaymath}
Note that $\mathbb{Z}[S_2]\text{-}\mathrm{sMod}=\mathbb{Z}[S_2]\text{-}\mathrm{Mod}$. 

\begin{proposition}\label{prop5.5}
The restriction induces an isomorphism of categories 
\begin{displaymath}
\mathrm{Res}^{\mathbb{Z}[S_2]}_{\mathbb{Z}_{\geq 0}[S_2]}:\mathbb{Z}[S_2]\text{-}\mathrm{Mod}\longrightarrow
\mathbb{Z}_{\geq 0}[S_2]\text{-}\mathrm{Mod}.
\end{displaymath}
\end{proposition}

\begin{proof}
Each  $\mathbb{Z}_{\geq 0}[S_2]$-module extends uniquely to  a $\mathbb{Z}[S_2]$-module 
by letting $-1$ act as the negation. If $M$ and $N$ are abelian groups and $\varphi:M\to N$
is a monoid homomorphism, then it is a group homomorphism. The claim follows.
\end{proof}

\subsection{Kazhdan-Lusztig version of $\mathbb{Z}_{\geq 0}[S_2]$}\label{s5.4}

Consider the element $\theta:=e+s$ in $\mathbb{Z}_{\geq 0}[S_2]$. 
We have $\theta^2=2\theta$.
Note that, for the ring
$\mathbb{Z}[S_2]$, the elements $e$ and $\theta$ form a basis of $\mathbb{Z}[S_2]$ over $\mathbb{Z}$,
called the {\em Kazhdan-Lusztig basis}, see \cite{KL}. 
Note, however, that the elements $e$ and $\theta$ are no longer
a basis of $\mathbb{Z}_{\geq 0}[S_2]$ over $\mathbb{Z}_{\geq 0}$. 
We denote by $\widehat{\mathbb{Z}_{\geq 0}[S_2]}$
the subsemiring of $\mathbb{Z}_{\geq 0}[S_2]$ generated by $e$ and $\theta$. Then 
$\widehat{\mathbb{Z}_{\geq 0}[S_2]}$ is free over $\mathbb{Z}_{\geq 0}$ with basis $e$ and $\theta$.

The results of this subsection give some feeling about  how the choice of a base semiring might 
affect the classification of simple semimodules. 

As $\widehat{\mathbb{Z}_{\geq 0}[S_2]}$ is a (unital) subsemiring of $\mathbb{Z}_{\geq 0}[S_2]$,
we have the restriction functor
\begin{displaymath}
\mathrm{Res}^{\mathbb{Z}_{\geq 0}[S_2]}_{\widehat{\mathbb{Z}_{\geq 0}[S_2]}}:
\mathbb{Z}_{\geq 0}[S_2]\text{-}\mathrm{sMod}\longrightarrow
\widehat{\mathbb{Z}_{\geq 0}[S_2]}\text{-}\mathrm{sMod}.
\end{displaymath}

\begin{lemma}\label{lem5.2}
Let $M$ be a simple $\mathbb{Z}_{\geq 0}[S_2]$-semimodule. Then its restriction to 
$\widehat{\mathbb{Z}_{\geq 0}[S_2]}$ is a simple $\widehat{\mathbb{Z}_{\geq 0}[S_2]}$-semimodule.
\end{lemma}

\begin{proof}
This follows directly from the fact that all simple  $\mathbb{Z}_{\geq 0}[S_2]$-semimodules are either
simple abelian groups or have just two elements, see Theorem~\ref{thm5.1}.
\end{proof}

\begin{lemma}\label{lem5.3}
The  commutative monoid $\mathbf{B}$ has the unique structure of a $\widehat{\mathbb{Z}_{\geq 0}[S_2]}$-semimodule
in which  $\theta$ acts as $0$.
\end{lemma}

\begin{proof}
We have $\theta^2=2\theta$. As $0^2=2\cdot 0$, the claim follows. 
\end{proof}

The $\widehat{\mathbb{Z}_{\geq 0}[S_2]}$-semimodule structure on $\mathbf{B}$ given by Lemma~\ref{lem5.3}
will be denoted $\mathbf{B}^{(0)}$.

The natural embedding of the semiring $\widehat{\mathbb{Z}_{\geq 0}[S_2]}$ into the ring 
$\mathbb{Z}[S_2]$ is a homomorphism of unital semirings. Therefore we have the corresponding restriction functor
\begin{displaymath}
\mathrm{Res}^{\mathbb{Z}[S_2]}_{\widehat{\mathbb{Z}_{\geq 0}[S_2]}}:
\mathbb{Z}[S_2]\text{-}\mathrm{sMod}\longrightarrow
\widehat{\mathbb{Z}_{\geq 0}[S_2]}\text{-}\mathrm{sMod}.
\end{displaymath}

\begin{proposition}\label{prop5.6}
The restriction induces an isomorphism of categories 
\begin{displaymath}
\mathrm{Res}^{\mathbb{Z}[S_2]}_{\widehat{\mathbb{Z}_{\geq 0}[S_2]}}:
\mathbb{Z}[S_2]\text{-}\mathrm{Mod}\longrightarrow
\widehat{\mathbb{Z}_{\geq 0}[S_2]}\text{-}\mathrm{Mod}.
\end{displaymath}
\end{proposition}

\begin{proof}
As mentioned above, $e$ and $\theta$ form a free $\mathbb{Z}$-basis of $\mathbb{Z}[S_2]$.
Each  $\widehat{\mathbb{Z}_{\geq 0}[S_2]}$-module extends uniquely to  a $\mathbb{Z}[S_2]$-module 
by letting $-1$ act as the negation. If $M$ and $N$ are abelian groups and $\varphi:M\to N$
is a monoid homomorphism, then it is a group homomorphism. The claim follows.
\end{proof}

\begin{theorem}\label{thm5.4}
Let $M\in \widehat{\mathbb{Z}_{\geq 0}[S_2]}\text{-}\mathrm{sMod}$. Then the following conditions are equivalent.
\begin{enumerate}[$($i$)$]
\item\label{thm5.4.1} $M$ is minimal.
\item\label{thm5.4.2} $M$ is elementary. 
\item\label{thm5.4.3} $M$ is simple. 
\item\label{thm5.4.4} $M$ is isomorphic to $\mathbf{B}^{(0)}$ or to the restriction of 
a simple $\mathbb{Z}_{\geq 0}[S_2]$-se\-mi\-mo\-du\-le.
\end{enumerate}
\end{theorem}

\begin{proof}
The implications \eqref{thm5.4.4}$\Rightarrow$\eqref{thm5.4.3}$\Rightarrow$\eqref{thm5.4.1} and
\eqref{thm5.4.4}$\Rightarrow$\eqref{thm5.4.3}$\Rightarrow$\eqref{thm5.4.2} follow from 
Lemmata~\ref{lem5.2} and \ref{lem5.3}.

In the case when $M$ is a $\widehat{\mathbb{Z}_{\geq 0}[S_2]}$-module, the implications 
\eqref{thm5.4.1}$\Rightarrow$\eqref{thm5.4.4} and \eqref{thm5.4.2}$\Rightarrow$\eqref{thm5.4.4}
follow from Theorem~\ref{thm5.1} and Propositions~\ref{prop5.5} and \ref{prop5.6}.

Let $M$ be a proper semimodule and $\tau:M\to M$ an endomorphism satisfying $\tau^2=2\tau$.
Let $m\in M$ be a non-invertible element. Then the submonoid $N$ of $M$ generated by 
$m$ and $\tau(m)$ is $\tau$-invariant as $\tau^2=2\tau$. 

Suppose first that $M$ is minimal and consider two cases.

{\bf Case~1.} Assume first that $\tau(m)$ is invertible.

Here we must have $M=N$ by minimality. If the order of $m$ were infinite, $M$ would have a proper
$\tau$-invariant submonoid generated by $2m$ and $2\tau(m)$, contradicting the minimality of $M$. 
Therefore the order of $m$ is finite. Let $p$ be an idempotent in the cyclic submonoid generated by $m$. 
Then $\tau(p)$, on the one hand, must be an idempotent as $\tau$ is an endomorphism, but, on the
other hand, must be invertible as $\tau(m)$ is. Hence $\tau(p)=0$. Therefore the submonoid 
$\{0,p\}$ is $\tau$-invariant and hence coincides with $M$ due to minimality. Thus $M\cong \mathbf{B}^{(0)}$.

{\bf Case~2.} Assume now that $\tau(m)$ is not invertible.

Consider the submonoid $N$ of $M$ generated by $\tau(m)$. This 
is $\tau$-invariant due to $\tau^2=2\tau$ and hence $M=N$ by minimality. If 
the order of $\tau(m)$ were infinite, then $M$ would have a proper $\tau$-invariant submonoid
generated by $2\tau(m)$,  contradicting the minimality of $M$. 
Therefore the order of $\tau(m)$ is finite. Let $p$ be an idempotent in the cyclic submonoid generated by 
$\tau(m)$. Then $\{0,p\}$ is a $\tau$-invariant submonoid of $M$ 
and hence coincides with $M$ due to minimality. Thus $M\cong \mathbf{B}$ with $\tau$ acting as the identity.

This establishes the implication \eqref{thm5.4.1}$\Rightarrow$\eqref{thm5.4.4}.

Suppose now that $M$ is elementary. Write $M=M_1\coprod M_2\coprod M_3$, where
\begin{gather*}
M_1:=\{\text{ all invertible elements of } M\};\\
M_2:=\{m\in M\,:m\not\in M_1,\tau(m)\in M_1\};\\
M_3:=M\setminus(M_1\cup M_2).
\end{gather*}
Note that $M_3$ is an ideal and that $|M_2\cup M_3|>0$ as $M$ is not a module. 
Also, note that $\tau^2=2\tau$ implies that 
$\tau(M_3)\subset M_3$. We naturally have two cases.

{\bf Case~A.} Assume that $M_3$ is empty.

In this case we have a $\tau$-stable congruence with congruence classes $M_1$ and $M_2$.
As $M$ is elementary, it follows that $|M_1|=|M_2|=1$ and hence $M\cong \mathbf{B}^{(0)}$.

{\bf Case~B.} Assume that $M_3$ is not empty.

In this case we have a $\tau$-stable congruence with congruence classes $M_1\cup M_2$ and $M_3$,
the quotient with respect to which is isomorphic to $\mathbf{B}$ with $\tau$ acting as the identity.
As $M$ is elementary, $M$ is isomorphic to $\mathbf{B}$ with $\tau$ acting as the identity.

This proves the implication \eqref{thm5.4.2}$\Rightarrow$\eqref{thm5.4.4} and completes the proof.
\end{proof}

\section{Extreme semimodules over non-negative real numbers}\label{s6}

\subsection{Construction}\label{s6.1}

The present section generalizes Examples~\ref{ex2.25} and \ref{ex2.26}. Let $\Bbbk$ be a subfield of $\mathbb{R}$.
Consider the semiring $\Bbbk_{\geq 0}$ of $\Bbbk$ consisting of all non-negative elements of $\Bbbk$. The map
$\psi:\Bbbk_{\geq 0}\to \mathbf{B}$ given by
\begin{displaymath}
\psi(x)=
\begin{cases}
0, & x=0,\\
1, & x\neq 0;
\end{cases}
\end{displaymath}
is a homomorphism of semirings. This equips $\mathbf{B}$ with the natural structure of a $\Bbbk_{\geq 0}$-semimodule
which we denote ${}_{\Bbbk_{\geq 0}}\mathbf{B}$. This semimodule is simple, since $|\mathbf{B}|=2$.

The left regular $\Bbbk_{\geq 0}$-semimodule ${}_{\Bbbk_{\geq 0}}\Bbbk_{\geq 0}$ is minimal because, for any non-zero element
$x\in \Bbbk_{\geq 0}$, we have $x^{-1}\in \Bbbk_{\geq 0}$ and hence $\Bbbk_{\geq 0}x=\Bbbk_{\geq 0}$.

The natural embedding $\Bbbk_{\geq 0}\hookrightarrow \Bbbk$ is a homomorphism of semirings. 
This equips $\Bbbk$ with the natural structure of a $\Bbbk_{\geq 0}$-semimodule
which we denote ${}_{\Bbbk_{\geq 0}}\Bbbk$. As $\Bbbk$ is a group with respect to addition, this semimodule is, in fact,
a module. 

\begin{lemma}\label{lem6.1}
The  $\Bbbk_{\geq 0}$-module ${}_{\Bbbk_{\geq 0}}\Bbbk$ is elementary.
\end{lemma}

\begin{proof}
We need to show that  ${}_{\Bbbk_{\geq 0}}\Bbbk$ contains no proper $\Bbbk_{\geq 0}$-invariant subgroups.
If $x\in \Bbbk$ is non-zero, then $\Bbbk_{\geq 0}x$ is the set of all elements in $\Bbbk$ of the same sign as
$x$. The additive subgroup of $\Bbbk$ generated by $\Bbbk_{\geq 0}x$ coincides with $\Bbbk$. The claim follows.
\end{proof}

\subsection{Classification}\label{s6.2}

\begin{theorem}\label{thm6.2}
The following holds:

\begin{enumerate}[$($a$)$]
\item\label{thm6.2.1} Up to isomorphism, the  $\Bbbk_{\geq 0}$-semimodule ${}_{\Bbbk_{\geq 0}}\mathbf{B}$ is 
the  only simple $\Bbbk_{\geq 0}$-se\-mi\-mo\-du\-le.
\item\label{thm6.2.2} Up to isomorphism, the  $\Bbbk_{\geq 0}$-semimodule ${}_{\Bbbk_{\geq 0}}\Bbbk_{\geq 0}$ 
is the  only minimal $\Bbbk_{\geq 0}$-se\-mi\-mo\-du\-le which is not simple.
\item\label{thm6.2.3} Up to isomorphism, the  $\Bbbk_{\geq 0}$-module ${}_{\Bbbk_{\geq 0}}\Bbbk$ 
is the only elementary $\Bbbk_{\geq 0}$-se\-mi\-mo\-du\-le which is not simple.
\end{enumerate}
\end{theorem}

This result generalizes naturally to the semiring of non-negative elements in any ordered field.

\subsection{Proof of Theorem~\ref{thm6.2}}\label{s6.3}

Let $M$ be a minimal $\Bbbk_{\geq 0}$-semimodule and $m\in M$ a non-zero element. Then the
assignment $1\mapsto m$ extends uniquely to a homomorphism from the free $\Bbbk_{\geq 0}$-semimodule
${}_{\Bbbk_{\geq 0}}\Bbbk_{\geq 0}$ to $M$ by the universal property of free semimodules. This
homomorphism must be surjective by the minimality of $M$. Therefore $M$ is a quotient of 
${}_{\Bbbk_{\geq 0}}\Bbbk_{\geq 0}$. To prove Claim~\eqref{thm6.2.2}, it remains to show that 
${}_{\Bbbk_{\geq 0}}\Bbbk_{\geq 0}$ has only one proper $\Bbbk_{\geq 0}$-congruence, namely the kernel
of the natural projection from ${}_{\Bbbk_{\geq 0}}\Bbbk_{\geq 0}$ to ${}_{\Bbbk_{\geq 0}}\mathbf{B}$.

Let $\sim$ be a $\Bbbk_{\geq 0}$-congruence on ${}_{\Bbbk_{\geq 0}}\Bbbk_{\geq 0}$. 
If $a\sim 0$, for some non-zero $a\in \Bbbk_{\geq 0}$, then $b\sim 0$ for any $b\in\Bbbk_{\geq 0}$
as $\Bbbk_{\geq 0}a=\Bbbk_{\geq 0}$ and $\Bbbk_{\geq 0}0=0$. 
Hence $\sim$ is the full relation $\Bbbk_{\geq 0}\times \Bbbk_{\geq 0}$ in this case.

Assume now that $a\sim b$, for two non-zero elements $a,b\in \Bbbk_{\geq 0}$ such that $a<b$. 
Let $c,d\in \Bbbk_{\geq 0}$ be two non-zero elements such that $c<d$. Choose a positive integer
$n$ such that $(1+n)(b-a)c>a(d-c)$. Such $n$ exists as both $(b-a)c$ and $a(d-c)$ are positive real
numbers. Then $a\sim (b+n(b-a))$ as $\sim$ is a congruence on ${}_{\Bbbk_{\geq 0}}\Bbbk_{\geq 0}$.
For the positive real numbers
\begin{displaymath}
\lambda=\frac{d-c}{(1+n)(b-a)}>0\quad\text{ and }\quad
p=\frac{(1+n)(b-a)c-a(d-c)}{(1+n)(b-a)}>0,
\end{displaymath}
we have $\lambda a+p=c$ and $\lambda(b+n(b-a))+p=d$. As $\sim$ is a $\Bbbk_{\geq 0}$-congruence,
this implies that  $c\sim d$ and shows that all positive elements of $\Bbbk_{\geq 0}$ belong to the same
$\sim$-equivalence class. Therefore $\sim$ is again either the full relation or 
coincides with the kernel of ${}_{\Bbbk_{\geq 0}}\Bbbk_{\geq 0}\tto{}_{\Bbbk_{\geq 0}}\mathbf{B}$.
This proves Claim~\eqref{thm6.2.2}.

Let us now prove Claim~\eqref{thm6.2.3}. Let $M$ be an elementary $\Bbbk_{\geq 0}$-semimodule.
Assume first that $M$ is not a module. Note that each non-zero element in $\Bbbk_{\geq 0}$
has a multiplicative inverse. Therefore every non-zero element in $\Bbbk_{\geq 0}$ acts
on $M$ by an automorphism, in particular, it preserves the set of all non-invertible elements in $M$.
Hence the equivalence relation on $M$ with two equivalence classes being the sets of all
invertible and all non-invertible elements is a $\Bbbk_{\geq 0}$-congruence. Consequently,
the equivalence classes must be singletons as $M$ is elementary. Hence 
$M\cong{}_{\Bbbk_{\geq 0}}\mathbf{B}$.

Now assume that $M$ is a $\Bbbk_{\geq 0}$-module. As usual, we denote by $\mathbf{Ab}$ the category of 
abelian groups. Consider the subsemiring $T$ of $\mathrm{End}_{\mathbf{Ab}}(M)$ generated by
the image of $\Bbbk_{\geq 0}$ under the module action and the sign change automorphism $(-1\cdot_{-})$.
Note that the latter automorphism commutes with all automorphisms of $M$. We claim that the difference 
between $T$ and the image of $\Bbbk_{\geq 0}$ are exactly all automorphisms of $M$ of the form
$x\circ (-1\cdot_{-})$, where $x\in \Bbbk_{\geq 0}$. For any $x\in \Bbbk_{\geq 0}$, we have the equality
\begin{equation}\label{eq6.2}
x\circ (\mathrm{id}_M)+ x\circ (-1\cdot_{-})=x\circ (\mathrm{id}_M+(-1)\cdot_{-})=x\circ (0\cdot{}_-)=0 
\end{equation}
of endomorphisms of $M$. Therefore the automorphisms of the form $x\circ (-1\cdot_{-})$ are exactly 
the additive  inverses to the automorphisms in the image of $\Bbbk_{\geq 0}$. Now we just need to show that
all these endomorphisms are closed under addition. By distributivity, we have
\begin{equation}\label{eq6.1}
x\circ (-1\cdot_{-})+y\circ (-1\cdot_{-})= (x+y)\circ (-1\cdot_{-}),
\end{equation}
for any $x,y\in \Bbbk_{\geq 0}$.
Let now $0<x<y$ be elements in $\Bbbk_{\geq 0}$. By \eqref{eq6.1}, we have 
\begin{displaymath}
y\circ (-1\cdot_{-})=x\circ (-1\cdot_{-})+(y-x)\circ (-1\cdot_{-})
\end{displaymath}
and hence, using \eqref{eq6.2}, we have 
\begin{displaymath}
\begin{array}{rcl} 
x\circ (\mathrm{id}_M)+y\circ (-1\cdot_{-})&=&x\circ (\mathrm{id}_M)+x\circ (-1\cdot_{-})+(y-x)\circ (-1\cdot_{-})\\
&=&(y-x)\circ (-1\cdot_{-}).
\end{array}
\end{displaymath}
Similarly one shows that 
\begin{displaymath}
y\circ (\mathrm{id}_M)+x\circ (-1\cdot_{-})=
(y-x)\circ (\mathrm{id}_M).
\end{displaymath}
This means that $T\cong \Bbbk$ is a field and $M$ is a vector space 
over this field. If $\mathrm{dim}_{\Bbbk}(M)=1$, then $M\cong {}_{\Bbbk_{\geq 0}}\Bbbk$.
If $\mathrm{dim}_{\Bbbk}(M)>1$, then any one-dimensional subspace of $M$ is a 
$\Bbbk_{\geq 0}$-submodule. This contradicts the fact that $M$ is elementary
and completes the proof of  Claim~\eqref{thm6.2.3}.

Claim~\eqref{thm6.2.1} follows from Claims~\eqref{thm6.2.2} and \eqref{thm6.2.3}.

\section{Some general results}\label{s4}

\subsection{Schur's lemma}\label{s4.1}

The following statement, cf. \cite[Lemma~2.6]{IRS}, 
is an analogue of Schur's lemma for representations of semirings.

\begin{lemma}\label{lem4.1}
Let $R$ be a semiring, $M,N\in R\text{-}\mathrm{sMod}$ and $\alpha\in\mathrm{Hom}_R(M,N)$
be a non-zero homomorphism. Then the following statements hold.
\begin{enumerate}[$($a$)$]
\item\label{lem4.1.1} If $M$ is elementary, then $\alpha$ is injective. 
\item\label{lem4.1.2} If $N$ is minimal, then $\alpha$ is surjective. 
\end{enumerate}
\end{lemma}

\begin{proof}
As the kernel of $\alpha$ is an $R$-congruence on $M$ and the image of $\alpha$
is a subsemimodule of $N$, the result follows directly from the definitions. 
\end{proof}

Lemma~\ref{lem4.1} has the following immediate consequences.

\begin{corollary}\label{cor4.2}
Every non-zero endomorphism of a simple $R$-semimodule is an isomorphism.
\end{corollary}

\begin{corollary}\label{cor4.3}
Every non-zero endomorphism of a finite elementary or a finite minimal $R$-semimodule is an isomorphism.
\end{corollary}

\subsection{Extreme modules}\label{s4.2}

\begin{proposition}\label{prop4.4}
Let $R$ be a semiring.
If $M$ is an $R$-module, then $M$ is minimal if and only if $M$ is simple.
\end{proposition}

\begin{proof}
We only need to show that minimal implies elementary. If $M$ is not elementary,
then we have a non-trivial $R$-congruence on $M$. By \cite[Subsection~6.2]{GaMa},
this is given by cosets with respect to a proper subgroup, say $N$. As $N$ contains $0$
and the latter element is fixed by the action of any $r\in R$, it follows that $N$
is $R$-invariant. Therefore $M$ is not minimal. The claim follows.
\end{proof}

\subsection{Generalities on extreme proper semimodules}\label{s4.25}

\begin{lemma}\label{lem4.45}
Let $R$ be a semiring and $M$ a minimal proper $R$-semimodule. Then $0$
is the only invertible element of $M$.
\end{lemma}

\begin{proof}
Every $r\in R$ acts on $M$ as an endomorphism, in particular, it maps invertible
elements to invertible elements. Therefore the set of all invertible elements is,
naturally, a subsemimodule of $M$. By the minimality of $M$ we hence have that
the set of invertible elements coincides either with $0$ or with $M$. As $M$
is proper, the second alternative is not possible.
\end{proof}

\begin{lemma}\label{lem4.46}
Let $R$ be a semiring and $M$ an elementary proper $R$-semimodule. Then $0$
is the only invertible element of $M$.
\end{lemma}

The following proof is similar to the proof of \cite[Proposition~1.2]{Il}.

\begin{proof}
Define an equivalence relation $\sim$ on $M$
as follows: $m\sim n$ if and only if $m=n+x$, for some invertible $x\in M$. 
We claim that this equivalence relation is an $R$-con\-gru\-ence. Indeed,
if $m\sim n$, then $m=n+x$, for some invertible $x$. 
Consequently, for any $k\in M$, we have $m+k=n+k+x$, that is $m+k\sim n+k$. 
Therefore $\sim$ is a congruence. Every $r\in R$ acts on $M$ as an endomorphism, 
in particular, it maps invertible elements to invertible elements. Hence
$m\sim n$ implies $m=n+x$. The latter, in turn, implies $r(m)=r(n)+r(x)$ where
$r(x)$ is invertible, that is, $r(m)\sim r(n)$. This
proves that $\sim$ is an $R$-congruence.

As $M$ is proper, $\sim$ does not coincide with the full relation on $M$.
Therefore  $\sim$ must be the equality relation due to the fact that $M$ is elementary. The claim follows.
\end{proof}

\subsection{The underlying monoid of a minimal proper semimodule}\label{s4.27}

\begin{proposition}\label{propmonoid}
Let $(M,+,0)$ be a non-zero finitely generated commutative monoid with the following properties.
\begin{enumerate}[$($a$)$]
\item\label{propmonoid.1} $0$ is the only invertible element of $M$.
\item\label{propmonoid.2} $M=kM$, for any $k\in\{2,3,\dots\}$.
\end{enumerate}
Then $M$ contains a non-zero idempotent.
\end{proposition}

\begin{proof}
Let $\mathbb{N}_{0}$ denote the additive monoid of all non-negative integers. As $M$ is finitely
generated, for some positive integer $n$, there is a surjective epimorphism
$\varphi:\mathbb{N}_{0}^n\tto M$. Due to condition~\eqref{propmonoid.1},
without loss of generality we may assume that the preimage of $0$
under $\varphi$ is a singleton. In particular, the images under $\varphi$ of all non-zero elements
are non-zero. Note that $\varphi$ cannot be an isomorphism as $\mathbb{N}_{0}^n$
does not satisfy condition~\eqref{propmonoid.2}.

Consider the element $a:=(1,0,\dots,0)\in\mathbb{N}_{0}^n$.
By \eqref{propmonoid.2}, for any $k\in\{2,3,\dots\}$, there is a non-zero element $a_k\in k\mathbb{N}_{0}^n$
such that $\varphi(a)=\varphi(a_k)$. If one of these $a_k$ has the form
$(x_1,x_2,\dots,x_n)$ with $x_1\neq 0$, then $\varphi((x_1-1,x_2,\dots,x_n))$ is a non-zero idempotent in 
$M$ since in this case we have  $(x_1-2,x_2,\dots,x_n)\in \mathbb{N}_{0}^n$ as $x_1\geq 2$ and therefore
\begin{displaymath}
\varphi((1,0,\dots,0))=\varphi((x_1,x_2,\dots,x_n)) 
\end{displaymath}
implies
\begin{displaymath}
\varphi((1,0,\dots,0)+(x_1-2,x_2,\dots,x_n))=\varphi((x_1,x_2,\dots,x_n)+(x_1-2,x_2,\dots,x_n)), 
\end{displaymath}
that is 
\begin{displaymath}
\varphi((x_1-1,x_2,\dots,x_n))=\varphi((2x_1-2,2x_2,\dots,2x_n))=2\varphi((x_1-1,x_2,\dots,x_n)). 
\end{displaymath}

It remains to consider the case when all $a_k$ have the form $(0,x_2,\dots,x_n)$. As $n$ is finite, there
exists $I\subset\{2,3,\dots,n\}$ such that, for infinitely many values of $k$, the non-zero entries of 
$a_k=(x_1,x_2,\dots,x_n)$  are exactly those indexed by elements in $I$.
Fix one such value of $k$ and let $a_k=(y_1,y_2,\dots,y_n)$.
Let $k'$ be another one of such values which, additionally, satisfies the condition that $k'$
is strictly greater than each of the $2y_i$. Then $a_{k'}-2a_k\in \mathbb{N}_{0}^n$ and just like in the previous
paragraph we obtain that $\varphi(a_{k'})=\varphi(a_k)$ implies that 
$\varphi(a_{k'}+a_{k'}-2a_k)=\varphi(a_k+a_{k'}-2a_k)$, that is, 
$2\varphi(a_{k'}-a_k)=\varphi(a_{k'}-a_k)$. Therefore
$\varphi(a_{k'}-a_k)$ is a non-zero idempotent, as claimed.
\end{proof}

\begin{corollary}\label{corminprosem}
Let $R$ be a finitely generated semiring and $M$ a minimal proper $R$-semimodule. 
Then every element of $M$ is an idempotent. 
\end{corollary}

\begin{proof}
As $R$ is finitely generated and $M$ is minimal, it follows that $M$ is finitely generated as a monoid.
By Lemma~\ref{lem4.45}, $M$ satisfies the condition in Proposition~\ref{propmonoid}\eqref{propmonoid.1}.
As $kM$ is a non-zero $R$-invariant submonoid of $M$, for every $k\in\{2,3,\dots\}$, the minimality
of $M$ implies that $M$ satisfies the condition in Proposition~\ref{propmonoid}\eqref{propmonoid.2}.
Therefore, by Proposition~\ref{propmonoid}, $M$ contains a non-zero idempotent.

Let $N$ denote the set of all idempotents in $M$. By the previous paragraph, $N$ contains at least two
elements. As $M$ is commutative, $N$ is a submonoid of $M$. As each element of $R$ acts as an endomorphism
of $M$, this action preserves $N$. Therefore $N$ is a non-zero $R$-subsemimodule of $M$.
From the minimality of $M$ we thus deduce that $N=M$, as claimed. 
\end{proof}

For elementary semimodules, analogous results can be found in
\cite[Sections~(15.27) and (15.28)]{Go} and
\cite[Proposition~1.2]{Il}.

\subsection{Proper semimodules of finite group semirings over $\mathbb{Z}_{\geq 0}$}\label{s4.3}

Let $G$ be a finite group and $R:=\mathbb{Z}_{\geq 0}[G]$ the corresponding group semiring over $\mathbb{Z}_{\geq 0}$.
Then the $\mathbb{Z}_{\geq 0}$-semimodule $\mathbf{B}$ extends to an $R$-semimodule be letting all
$g\in G$ act as the identity. We will denote this $R$-semimodule by ${}_R\mathbf{B}$.

\begin{proposition}\label{prop4.5}
Let $G$ be a finite group and $M$ a proper $R:=\mathbb{Z}_{\geq 0}[G]$-semimodule. Then the 
following conditions are equivalent.
\begin{enumerate}[$($i$)$]
\item\label{thm4.5.1} $M$ is minimal.
\item\label{thm4.5.2} $M$ is elementary. 
\item\label{thm4.5.3} $M$ is simple. 
\item\label{thm4.5.4} $M$ is isomorphic to ${}_R\mathbf{B}$. 
\end{enumerate}
\end{proposition}

\begin{proof}
The implication \eqref{thm4.5.4}$\Rightarrow$\eqref{thm4.5.3} follows from the fact that 
$|\mathbf{B}|=2$. The implications \eqref{thm4.5.3}$\Rightarrow$\eqref{thm4.5.2}
and \eqref{thm4.5.3}$\Rightarrow$\eqref{thm4.5.1} follow directly from the definitions.
 
To prove the implication  \eqref{thm4.5.1}$\Rightarrow$\eqref{thm4.5.4}, assume that $M$ is minimal
and let $m\in M$ be a non-invertible element (this $m$ exists as $M$ is assumed to be a proper semimodule). 
Then $g(m)$ is non-invertible, for all $g\in G$ and hence the element
$n:=\sum_{g\in G}g(m)$ is both, non-invertible (as non-invertible elements of $M$ form an ideal) and
$G$-in\-va\-ri\-ant. As $M$ is minimal, it is generated by $0$ and $n$, in particular, the action of $G$
on $M$ is trivial. Now the fact that $M\cong {}_R\mathbf{B}$ follows from Theorem~\ref{thm3.2},
proving the implication \eqref{thm4.5.1}$\Rightarrow$\eqref{thm4.5.4}.

To prove the implication  \eqref{thm4.5.2}$\Rightarrow$\eqref{thm4.5.4}, assume that $M$ is elementary.
Decompose $M=M_0\coprod M_1$, where $M_0$ is the set of invertible elements of $M$ and  
$M_1$ is the set of non-invertible elements of $M$. Note that $M_0\neq\varnothing$ as $0\in M_0$
and $M_1\neq\varnothing$ as $M$ is assumed to be a proper semimodule. As $G$ is a group, the action of
$G$ preserves both $M_0$ and $M_1$. This implies that the equivalence relation on $M$ with two classes
$M_0$ and $M_1$ is an $R$-congruence. As $M$ is assumed to be elementary, it follows that $|M_0|=|M_1|=1$
and thus $M\cong {}_R\mathbf{B}$, establishing the implication  \eqref{thm4.5.2}$\Rightarrow$\eqref{thm4.5.4}.
This completes the proof.
\end{proof}

\subsection{$\mathbb{Z}_{\geq 0}[G]$-modules}\label{s4.4}

Let $G$ be a finite group and $R:=\mathbb{Z}_{\geq 0}[G]$ the corresponding group semiring over $\mathbb{Z}_{\geq 0}$.
The following statement is a generalization of Proposition~\ref{prop5.5}.

\begin{proposition}\label{prop4.5a}
The restriction functor from $\mathbb{Z}[G]$-Mod to $\mathbb{Z}_{\geq 0}[G]$-Mod is an isomorphism of categories. 
\end{proposition}

\begin{proof}
Just like in the proof of  Proposition~\ref{prop5.5}, the inverse of this restriction is given by the unique 
extension of a $\mathbb{Z}_{\geq 0}[G]$-module structure to a $\mathbb{Z}[G]$-module structure by defining
the action of $-1\in \mathbb{Z}$ as the negation on the module.
\end{proof}

\subsection{$\mathbb{Z}[G]$-modules}\label{s4.5}

Let $G$ be a finite group. The following statement is a version of a result of P.~Hall, see \cite{Ha}.

\begin{proposition}\label{prop4.6}
Let $M$ be a simple $\mathbb{Z}[G]$-module. Then there is a prime $p\in\mathbb{Z}$ such that $pM=0$. 
\end{proposition}

\begin{proof}
As $M$ is simple, it is finitely generated over $\mathbb{Z}[G]$. As $G$ is finite, $\mathbb{Z}[G]$
is of finite rank over $\mathbb{Z}$. Therefore $M$ is finitely generated over $\mathbb{Z}$.

Assume first that $M$ is torsion-free over $\mathbb{Z}$. Then the subgroup $2M\neq M$ is $G$-invariant
and hence is a submodule. This contradicts simplicity of $M$. Therefore the torsion subgroup of
$M$ is non-zero. As all elements of $G$ act as automorphisms of $M$, they preserve the torsion subgroup
and hence this torsion subgroup is, in fact, a submodule. Due to simplicity of $M$, it follows
that $M$ is torsion.

Consequently, there is a prime $p\in\mathbb{Z}$ such that $\mathrm{Ker}(p\cdot{}_-)$ on $M$ is non-zero.
Again, as all elements of $G$ act as automorphisms of $M$, they preserve $\mathrm{Ker}(p\cdot{}_-)$ which
means that $M=\mathrm{Ker}(p\cdot{}_-)$ due to simplicity of $M$.
\end{proof}

As an immediate corollary from Proposition~\ref{prop4.6}, we have the following statement which reduces
classification of simple $\mathbb{Z}[G]$-modules to that of simple $\mathbb{Z}_p[G]$-modules, for all $p$.

\begin{corollary}\label{cor4.7}
Let $M$ be a simple $\mathbb{Z}[G]$-module. Then there is a prime $p\in\mathbb{Z}$ such that 
$M$ is a pullback, via $\mathbb{Z}[G]\tto\mathbb{Z}_p[G]$, of a simple $\mathbb{Z}_p[G]$-module.
\end{corollary}

\subsection{Extreme $\mathbb{Z}_{\geq 0}[G]$-modules}\label{s4.6}

\begin{theorem}\label{thm4.8}
Let $G$ be a finite group and $M$ a $\mathbb{Z}_{\geq 0}[G]$-module. 
Then the  following conditions are equivalent.
\begin{enumerate}[$($i$)$]
\item\label{thm4.8.1} $M$ is minimal.
\item\label{thm4.8.2} $M$ is elementary. 
\item\label{thm4.8.3} $M$ is simple. 
\item\label{thm4.8.4} $M$ is the restriction of a simple $\mathbb{Z}[G]$-module. 
\end{enumerate}
\end{theorem}

\begin{proof}
Assume \eqref{thm4.8.4}. Then $M$ is the restriction of a simple $\mathbb{Z}_p[G]$-module by 
Corollary~\ref{cor4.7}. Note that the restriction of the map $\mathbb{Z}[G]\tto\mathbb{Z}_p[G]$
to the subset $\mathbb{Z}_{\geq 0}[G]$ of the domain
remains surjective. Hence we have \eqref{thm4.8.3} by definition and this  also implies
\eqref{thm4.8.1} and \eqref{thm4.8.2} as special cases.

Assume \eqref{thm4.8.1} and consider $M$ as a $\mathbb{Z}[G]$-module via Proposition~\ref{prop4.5a}.
Now, the minimality of $M$ as  a $\mathbb{Z}_{\geq 0}[G]$-module implies simplicity of 
$M$ as a $\mathbb{Z}[G]$-module since any $\mathbb{Z}[G]$-sub\-mo\-du\-le of $M$ would also be
a $\mathbb{Z}_{\geq 0}[G]$-submodule. Hence we have \eqref{thm4.8.4}. Similarly 
\eqref{thm4.8.2} implies  \eqref{thm4.8.4}. The claim follows.
\end{proof}

\section{Extreme $\mathbb{Z}_{\geq 0}[S_3]$-semimodules}\label{s7}

\subsection{The symmetric group $S_3$}\label{s7.1}

In this section we study extreme semimodules over various semirings related to the 
symmetric group $S_3$ of permutations of the set $\{1,2,3\}$. We let $s$ to be the transposition $(1,2)$
and $t$ be the transposition $(2,3)$. Then $S_3=\{e,s,t,st,ts,w_0:=sts=tst\}$.

The Hasse diagram of the {\em Bruhat order} $\preceq$ on $S_3$ has the form
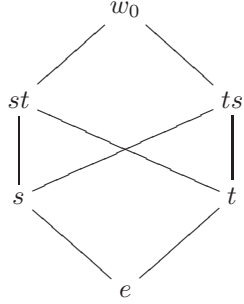
\begin{displaymath}
\xymatrix{ 
&w_0\ar@{-}[dl]\ar@{-}[dr]&\\
st\ar@{-}[d]\ar@{-}[drr]&&ts\ar@{-}[d]\ar@{-}[dll]\\
s\ar@{-}[dr]&&t\ar@{-}[dl]\\
&e&
}
\end{displaymath}
The {\em Kazhdan-Lusztig} basis $\{\underline{w}:w\in S_3\}$, cf. \cite{KL}, is defined via 
\begin{displaymath}
\underline{w}:=\sum_{x\preceq w}x.
\end{displaymath}
Concretely, we have 
\begin{gather*}
\underline{e}=e,\quad  
\underline{s}=e+s,\quad  
\underline{t}=e+t,\quad  \\
\underline{st}=e+s+t+st,\quad  
\underline{ts}=e+s+t+ts,\quad  
\underline{w_0}=e+s+t+st+ts+w_0.  
\end{gather*}
The multiplication table $(a,b)\mapsto ab$ in this basis is given by
\begin{equation}\label{eq-71}
\begin{array}{c||c|c|c|c|c|c}
a\backslash b& \underline{e}& \underline{s}& \underline{t}& \underline{st}& \underline{ts}& \underline{w_0}\\
\hline\hline
\underline{e}& \underline{e}& \underline{s}& \underline{t}& \underline{st}& \underline{ts}& \underline{w_0}\\
\hline
\underline{s}& \underline{s}& 2\,\underline{s}& \underline{st}& 2\,\underline{st}& 
\underline{s}+\underline{w_0}& 2\,\underline{w_0}\\
\hline
\underline{t}& \underline{t}& \underline{ts}& 2\,\underline{t}& 
\underline{t}+\underline{w_0}& 2\,\underline{ts}& 2\,\underline{w_0}\\
\hline
\underline{st}& \underline{st}& \underline{s}+\underline{w_0}& 
2\,\underline{st}& \underline{st}+2\,\underline{w_0}& 2\,\underline{s}+2\,\underline{w_0}& 4\,\underline{w_0}\\
\hline
\underline{ts}& \underline{ts}& 2\,\underline{ts}& \underline{t}+\underline{w_0}& 
2\,\underline{t}+2\,\underline{w_0}& \underline{ts}+2\,\underline{w_0}& 4\,\underline{w_0}\\
\hline
\underline{w_0}& \underline{w_0}& 2\,\underline{w_0}& 2\,\underline{w_0}& 
4\,\underline{w_0}& 4\,\underline{w_0}& 6\,\underline{w_0}\\
\end{array}
\end{equation}

In this section we denote by $R$ the $\mathbb{Z}_{\geq 0}$-subsemiring of $\mathbb{Z}[S_3]$
spanned by the set $\{\underline{w}:w\in S_3\}$. In the above multiplication table we see that 
all structure constants in the Kazhdan-Lusztig basis are non-negative integers. It follows
that $R$ is free as a $\mathbb{Z}_{\geq 0}$-semimodule with basis $\{\underline{w}:w\in S_3\}$.

\subsection{Extreme  $\mathbb{Z}_{\geq 0}[S_3]$-modules}\label{s7.2}

By Subsection~\ref{s4.4} all extreme $\mathbb{Z}_{\geq 0}[S_3]$-modules are restrictions to
$\mathbb{Z}_{\geq 0}[S_3]$ of simple $\mathbb{Z}_{p}[S_3]$-modules, where $p$ is a prime.
Let us recall all such modules. There are three simple modules for every $p\neq 2,3$ and
there are two simple modules for $p=2,3$.

For any prime $p$, we have the {\em trivial} module $\mathbb{Z}_{p}$ on which each $w\in S_3$
acts as the identity.

For any prime $p>2$, we have the {\em sign} module $\mathbb{Z}_{p}$ on which each $w\in S_3$
acts as the multiplication with $\mathrm{sign}(w)$. This module is also defined for $p=2$ but
in this case it coincides with the trivial module.

For $p\neq 3$, we have the $2$-dimensional module $\mathbb{Z}_{p}\oplus \mathbb{Z}_{p}$ on which
$s$ and $t$ act as the linear transformations given by the matrices
\begin{displaymath}
\left(\begin{array}{cc}0&1\\1&0\end{array}\right),\quad
\left(\begin{array}{cc}1&p-1\\0&p-1\end{array}\right),
\end{displaymath}
respectively. This module is also defined for $p=3$, however, in the latter case it is not simple
as the linear span of $\left(\begin{array}{c}1\\2\end{array}\right)$ forms a submodule.

\subsection{Extreme proper $\mathbb{Z}_{\geq 0}[S_3]$-semimodules}\label{s7.3}

Extreme proper $\mathbb{Z}_{\geq 0}[S_3]$-se\-mi\-mo\-du\-les are classified by Proposition~\ref{prop4.5}.
In fact, there is only one such semimodule, namely the semimodule 
${}_{\mathbb{Z}_{\geq 0}[S_3]}\mathbf{B}$ on which all $w\in S_3$ act 
as the identity. 

\subsection{Extreme $R$-modules}\label{s7.4}

\begin{proposition}\label{prop7.1}
The restriction functor from $\mathbb{Z}[S_3]$-Mod to $R$-Mod is an isomorphism of categories. 
\end{proposition}

\begin{proof}
Just like in the proof of  Proposition~\ref{prop5.5}, the inverse of this restriction is given by the unique 
extension of an $R$-module structure to a $\mathbb{Z}[S_3]$-module structure by defining
the action of $-1\in \mathbb{Z}$ as the negation on the module.
\end{proof}

Consequently, extreme $R$-modules are exactly the restrictions of simple $\mathbb{Z}[S_3]$-mo\-du\-les,
see Subsection~\ref{s7.2} for an explicit description of the latter.

\subsection{Some extreme proper $R$-semimodules}\label{s7.5}

We denote by $M_1$ the $R$-se\-mi\-mo\-du\-le with the underlying monoid $\mathbf{B}$ on which all 
$\underline{w}$, where $w\in S_3$, act as the identity.

We denote by $M_2$ the $R$-semimodule with the underlying monoid $\mathbf{B}$ on which all 
$\underline{w}$, where $e\neq w\in S_3$, act as zero.

We denote by $M_3$ the $R$-semimodule with the underlying monoid $\mathbf{B}$ on which all 
$\underline{w}$, where $w_0\neq w\in S_3$, act as the identity while $\underline{w_0}$ acts as zero.

We denote by $M_4$ the $R$-semimodule with the underlying monoid 
being the additive version of $\mathbf{B}\oplus\mathbf{B}$ on which 
the action of $\underline{w}$, where $w\in S_3$, is defined as follows:
\begin{itemize}
\item $\underline{e}$ acts as the identity;
\item $\underline{w_0}$ acts as zero;
\item $\underline{s}$ and $\underline{st}$ act by sending all non-zero elements to $(1,0)$;
\item $\underline{t}$ and $\underline{ts}$ act by sending all non-zero elements to $(0,1)$.
\end{itemize}

The fact that $M_1$ and $M_2$ are $R$-semimodules is straightforward.
That $M_3$ and $M_4$ are $R$-semimodules is easily checked using \eqref{eq-71}.
The $R$-semimodules $M_1$, $M_2$ and $M_3$ are simple. 
The $R$-semimodule $M_4$ is minimal but not elementary. Indeed, sending $(i,j)$ to $i+j$ 
defines a surjective homomorphism from $M_4$ to $M_3$ which has a non-trivial kernel.

We denote by $M_5$ the quotient of $M_4$ by the $R$-congruence which identifies
$(1,0)$ and $(1,1)$. We denote by $M_6$ the quotient of $M_4$ by the $R$-congruence 
which identifies $(0,1)$ and $(1,1)$. Then both $M_5$ and $M_6$ are minimal by
Proposition~\ref{prop2.2}.

Denote by $M_7$ the $R$-semimodule with the underlying monoid 
being the additive version of $\mathbf{B}\oplus\mathbf{B}$ on which 
the action of $\underline{w}$, where $w\in S_3$, is defined as follows:
\begin{itemize}
\item $\underline{e}$ acts as the identity;
\item $\underline{w_0}$ acts as zero;
\item $\underline{s}$ and $\underline{ts}$ act as follows: 
\begin{equation}\label{eq733}
\begin{array}{c||c|c|c|c}
m&(0,0)&(0,1)&(1,0)&(1,1)\\ \hline
\underline{s}(m)=\underline{ts}(m)&(0,0)&(0,0)&(1,1)&(1,1)\\ 
\end{array}
\end{equation}
\item $\underline{t}$ and $\underline{st}$ act as follows: 
\begin{equation}\label{eq734}
\begin{array}{c||c|c|c|c}
m&(0,0)&(0,1)&(1,0)&(1,1)\\ \hline
\underline{t}(m)=\underline{st}(m)&(0,0)&(1,1)&(0,0)&(1,1)\\ 
\end{array}
\end{equation}
\end{itemize}
We denote by $M_8$ the $R$-subsemimodule of $M_7$ with the underlying monoid 
consisting of the elements $(0,0)$, $(1,0)$ and $(1,1)$.
We denote by $M_9$ the $R$-subsemimodule of $M_7$ with the underlying monoid 
consisting of the elements $(0,0)$, $(0,1)$ and $(1,1)$.

\begin{lemma}\label{lem7.05}
The above defines on $M_7$, $M_8$ and $M_9$ the structures of elementary $R$-semimodules. 
\end{lemma}

\begin{proof}
Let us start with $M_7$. First of all we claim that all $\underline{w}$ are endomorphisms
of $\mathbf{B}\oplus\mathbf{B}$. That can be checked directly.
Further, taking into account that $\underline{w_0}$ acts as zero and 
that the monoid underlying $M_7$ consists of idempotents, the multiplication table
\eqref{eq-71} that we have to check takes the following form:
\begin{equation}\label{eqneq5}
\begin{array}{c||c|c|c|c|c}
a\backslash b& \underline{e}& \underline{s}& \underline{t}& \underline{st}& \underline{ts}\\
\hline\hline
\underline{e}& \underline{e}& \underline{s}& \underline{t}& \underline{st}& \underline{ts}\\
\hline
\underline{s}& \underline{s}& \underline{s}& \underline{st}& \underline{st}& 
\underline{s}\\
\hline
\underline{t}& \underline{t}& \underline{ts}& \underline{t}& 
\underline{t}& \underline{ts}\\
\hline
\underline{st}& \underline{st}& \underline{s}& 
\underline{st}& \underline{st}& \underline{s}\\
\hline
\underline{ts}& \underline{ts}& \underline{ts}& \underline{t}& 
\underline{t}& \underline{ts}
\end{array}
\end{equation}
It is straightforward to check that our definitions agree with this table
(here it is helpful to observe that both $\underline{s}$ and $\underline{t}$ are defined to act as
the identity transformations on all elements which appear in the 
lower rows of \eqref{eq733} and \eqref{eq734}). Hence $M_7$
is, indeed, an $R$-semimodule. Consequently, $M_8$ and $M_9$ are $R$-semimodules as well
as they are closed both with respect to addition and with respect to the action of $R$.

Assume that $\sim$ is an  $R$ congruence on $M_7$ different from the identity relation. 
If $(0,0)\sim (1,0)$, then $(0,0)\sim (1,1)$ by applying $\underline{s}$. Adding 
$(0,1)$ to $(0,0)\sim (1,1)$, we get $(0,1)\sim (1,1)$ and thus $\sim$ is the full relation.
If $(0,0)\sim (0,1)$, then $(0,0)\sim (1,1)$ by applying $\underline{t}$. Adding 
$(1,0)$ to $(0,0)\sim (1,1)$, we get $(1,0)\sim (1,1)$ and thus $\sim$ is the full relation.
If $(0,0)\sim (1,1)$, then $(0,1)\sim (1,1)$ by the above and, similarly,
$(1,0)\sim (1,1)$. Again, $\sim$ is the full relation.

Assume now that $(0,0)$ is a singleton equivalence class for $\sim$. If $(1,0)\sim(1,1)$,
applying $\underline{t}$ we get $(0,0)\sim(1,1)$, a contradiction. By a similar argument,
$(0,1)\sim(1,1)$ is not possible either. It remains to check the case when $(1,0)\sim(0,1)$.
However, again, applying $\underline{t}$ we get $(0,0)\sim(1,1)$, a contradiction.
This shows that the only $R$-congruences on $M_7$ are the equality relation and the full relation.
Therefore $M_7$ is elementary. The fact that $M_8$ and $M_9$ are elementary follows
from the above arguments for $M_7$. This proves our lemma.
\end{proof}

Below we illustrate the underlying meet semi-lattices of $M_4$ (on the left) and $M_7$ (on the right)
via their corresponding Hasse diagram depicted by the solid lines
with the action of $\underline{s}$ depicted by the dashed arrows and
the action of $\underline{t}$ depicted by the dotted arrows.
\begin{equation}\label{eqnewfor9}
\xymatrix{
&(0,0)\ar@{-}[dr]\ar@{-}[dl]\ar@{-->}@(l,ul)[]\ar@{.>}@(r,ur)[]&\\
(1,0)\ar@{-}[dr]\ar@{.>}@/^/[rr]\ar@{-->}@(l,ul)[]
&&(0,1)\ar@{-}[dl]\ar@{-->}@/^/[ll]\ar@{.>}@(r,ur)[]\\
&(1,1)\ar@{-->}@/^/[ul]\ar@{.>}@/_/[ur]&\\
}\qquad\qquad
\xymatrix{
&(0,0)\ar@{-}[dr]\ar@{-}[dl]\ar@{-->}@(l,ul)[]\ar@{.>}@(r,ur)[]&\\
(1,0)\ar@{-}[dr]\ar@{-->}@/_/[dr]\ar@{.>}@/^/[ur]
&&(0,1)\ar@{-}[dl]\ar@{-->}@/_/[ul]\ar@{.>}@/^/[dl]\\
&(1,1)\ar@{.>}@(l,dl)[]\ar@{-->}@(r,dr)[]&\\
}
\end{equation}

\subsection{Classification of extreme proper $R$-semimodules}\label{s7.6}

\begin{theorem}\label{thm7.2}
The $R$-semimodules $M_1$, $M_2$, $M_3$, $M_4$, $M_5$ and $M_6$
are the only minimal proper $R$-semimodules. 
\end{theorem}

\begin{proof}
Let $M$ be a minimal proper $R$-semimodule. Then every element of $M$ is an idempotent by 
Corollary~\ref{corminprosem}.

Assume first that $M$ contains some non-invertible
element $m$ such that $n:=\underline{w_0}(m)$ is non-invertible. Then from \eqref{eq-71} it follows
that the submonoid $N=\{0,n\}$ of $M$  is $R$-invariant
and hence $M\cong M_1$ by minimality of $M$.

Next we note that \eqref{eq-71} implies that $\underline{w_0}(M)$ is $R$-invariant. 
After the previous paragraph, we may from now on assume that $\underline{w_0}(M)$ consists of 
invertible elements of $M$. We have $|\underline{w_0}(M)|=1$ by Lemma~\ref{lem4.45} which implies
that $\underline{w_0}$ acts on $M$ by multiplication with zero.

Assume that there is a non-invertible $m\in M$ such that $\underline{s}(m)$ is non-invertible.
Then $\underline{ts}(m)$ must be non-invertible for otherwise $\underline{s}\cdot \underline{ts}(m)=
\underline{w_0}(m)+\underline{s}(m)=\underline{s}(m)$ would be invertible. Consider 
the submonoid $N$ of $M$ generated by $\underline{s}(m)$ and $\underline{ts}(m)$. We have
\begin{displaymath}
\underline{s}\cdot \underline{s}(m)=2\,\underline{s}(m),\quad
\underline{s}\cdot \underline{ts}(m)=\underline{s}(m),\quad
\underline{t}\cdot \underline{s}(m)=\underline{ts}(m),\quad
\underline{t}\cdot \underline{ts}(m)=2\,\underline{ts}(m). 
\end{displaymath}
This means that $N$ is $R$-invariant and hence $N=M$ by the minimality of $M$.
As every element in $M$ is an idempotent, we have 
\begin{displaymath}
M=\{0,\underline{s}(m),\underline{ts}(m),\underline{s}(m)+\underline{ts}(m)\}. 
\end{displaymath}
If  $\underline{s}(m)=\underline{ts}(m)$, then $M=\{0,\underline{s}(m)\}$ and $M\cong M_3$.
If $\underline{s}(m)\neq\underline{ts}(m)$ and $|M|=4$,
then $M$ is isomorphic to $M_4$. If  $\underline{s}(m)\neq\underline{ts}(m)$ and
$|M|=3$, then $M$ is isomorphic to
either $M_5$ or $M_6$, depending on whether we have $\underline{s}(m)=\underline{s}(m)+\underline{ts}(m)$ 
or $\underline{ts}(m)=\underline{s}(m)+\underline{ts}(m)$.

Finally, assume that $\underline{s}(M)$ consists of invertible elements. By symmetry,
$\underline{t}(M)$ consist of invertible elements as well. 
From Lemma~\ref{lem4.45} we thus get that
both $\underline{s}$ and $\underline{t}$ act on $M$ as multiplication with $0$. Hence
both $\underline{st}$ and $\underline{ts}$ act on $M$ as multiplication with $0$ as well.
Let $m\in M$ be non-invertible (which exists as $M$ is proper). Then the submonoid 
$N=\{0,m\}$ of $M$ is $R$-invariant and hence, due to the minimality of $M$, we have $M\cong M_2$.
\end{proof}

As an immediate corollary from Theorem~\ref{thm7.2}, we have:

\begin{corollary}\label{thm7.4}
The $R$-semimodules $M_1$, $M_2$ and $M_3$ are the only simple proper $R$-semimodules. 
\end{corollary}

\begin{theorem}\label{thm7.3}
The $R$-semimodules $M_1$, $M_2$, $M_3$, $M_7$, $M_8$ and $M_9$ are the only elementary proper $R$-semimodules. 
\end{theorem}

For a set $X$ and a subset $Y$ of $X$, we denote by $\sim_Y$ the equivalence relation on
$X$ given, for $x,x'\in X$, by
\begin{displaymath}
x\sim_Y x'\quad\text{ if and only if }\quad 
x=x'\text{ or } x,x'\subset Y. 
\end{displaymath}

\begin{proof}
Let $M$ be an elementary proper $R$-semimodule. 
By Lemma~\ref{lem4.46}, the element $0$ is the only invertible element of $M$.
Consider the set 
\begin{displaymath}
I=\{m\in M\,:\, \underline{w_0}(m) \text{ is not invertible}\}. 
\end{displaymath}
For $m\in I$ and $n\in M$ we have $\underline{w_0}(m+n)=\underline{w_0}(m)+\underline{w_0}(n)$
is not invertible as $\underline{w_0}(m)$ is not invertible and, due to commutativity of $M$,
all non-invertible elements of $M$ form an ideal. This means that $I$ is an ideal of $M$.
For $w\in S_3$ and $m\in I$, we have 
\begin{displaymath}
\underline{w_0}(\underline{w}(m))=(\underline{w_0}\underline{w})(m)=
k\underline{w_0}(m),
\end{displaymath}
for some $k\in\mathbb{Z}_{>0}$, using \eqref{eq-71}. This implies that $I$ is $R$-invariant.
Consequently, $\sim_I$ is an $R$-congruence which is, moreover, different from the full
relation as $0\not\in I$. As  $M$ is elementary, it follows that $|I|\leq 1$.

Consider first the case $|I|=1$, say, $I=\{h\}$. Then $h+h=h$ and $\underline{w_0}(h)=h$ by construction.
The computation from the previous paragraph also implies $\underline{w}(h)=h$, for any $w\in S_3$.
Let $J:=M\setminus\{h\}$. Then $\underline{w_0}(m)=0$, for every $m\in J$, in particular, 
$m+n\in J$, for any $m,n\in J$. Furthermore, the computation from the previous paragraph implies 
$\underline{w}(m)\in J$, for any $w\in S_3$ and $m\in J$. For $m\in J$, we have 
\begin{displaymath}
\underline{w_0}(h+m)=\underline{w_0}(h)+\underline{w_0}(m)=h+\underline{w_0}(m)=h, 
\end{displaymath}
which means that $h+m=h$. Consequently, $\sim_J$ is an $R$-congruence on $M$ which is different from the 
full relation. Hence $|J|\leq 1$ and, since $0\in J$, we get $J=\{0\}$.
In this case we have $M\cong M_1$.

Now we consider  the case $|I|=0$. In this case $\underline{w_0}$ acts on $M$ as multiplication by $0$.
Consider the sets
\begin{gather*}
B_{00}:=\{m\in M\setminus\{0\}\,:\, \underline{s}(m) = 0 \text{ and }\underline{t}(m)= 0\},\\ 
B_{01}:=\{m\in M\,:\, \underline{s}(m) = 0 \text{ and }\underline{t}(m)\neq 0\},\\ 
B_{10}:=\{m\in M\,:\, \underline{s}(m) \neq  0 \text{ and }\underline{t}(m)= 0\},\\ 
B_{11}:=\{m\in M\,:\, \underline{s}(m) \neq 0 \text{ and }\underline{t}(m)\neq 0\}. 
\end{gather*}
Let $\sim$ denote the equivalence relation on $M$ with equivalence classes
$\{0\}$, $B_{00}$, $B_{01}$, $B_{10}$ and $B_{11}$.

\begin{lemma}\label{lem7.11}
The equivalence relation $\sim$ is an $R$-congruence on $M$. 
\end{lemma}

\begin{proof}
Assume that $n\in B_{00}$ and $m\in M$. If $m\in B_{00}$, then we have
\begin{displaymath}
\underline{s}(n+m)=\underline{s}(n)+\underline{s}(m)=0+0=0 
\end{displaymath}
and, similarly, $\underline{t}(n+m)=0$. Hence $n+m\in B_{00}$.
If $m\in B_{01}$, then we have $\underline{s}(n+m)=0$ by the previous computation and
\begin{displaymath}
\underline{t}(n+m)=\underline{t}(n)+\underline{t}(m)=0+\underline{t}(m)=\underline{t}(m)\neq 0. 
\end{displaymath}
Hence $n+m\in B_{01}$. Analogously one shows that $m\in B_{10}$ implies $n+m\in B_{10}$
and, further, that $m\in B_{11}$ implies $n+m\in B_{11}$.

Proceeding with $n\in B_{10}$, $n\in B_{01}$ and  $n\in B_{11}$, one checks that $\sim$ is
a congruence on $M$ the quotient by which is a commutative monoid with the following addition table:
\begin{equation}\label{eq720}
\begin{array}{c||c|c|c|c|c}
+& 0 & B_{00} & B_{10}& B_{01}& B_{11}\\
\hline\hline
0& 0 & B_{00} & B_{10}& B_{01}& B_{11}\\
\hline
B_{00}& B_{00} & B_{00} & B_{10}& B_{01}& B_{11}\\
\hline
B_{10}& B_{10} & B_{10} & B_{10}& B_{11}& B_{11}\\
\hline
B_{01}& B_{01} & B_{01} & B_{11}& B_{01}& B_{11}\\
\hline
B_{11}& B_{11} & B_{11} & B_{11}& B_{11}& B_{11}  
\end{array}
\end{equation}

Now let us check that $\sim$ is $R$-invariant. That  $B_{00}$ is sent by $R$ to $0$
follows directly from the definition. Consider $B_{10}$. Then $\underline{t}$ sends it
to $0$. We claim that $\underline{s}$ sends it to $B_{11}$. Indeed, let $m\in B_{10}$.
Then $\underline{s}(\underline{s}(m))=2\underline{s}(m)\neq 0$. 
Assume $\underline{t}(\underline{s}(m))=0$. Then, as $\underline{w_0}$ acts on $M$ as zero, we have
\begin{equation}\label{eq721}
0= \underline{s}(\underline{ts}(m))=(\underline{s}\cdot \underline{ts})(m)=
(\underline{w_0}+\underline{s})(m)=\underline{w_0}(m)+\underline{s}(m)=0+\underline{s}(m)=\underline{s}(m)\neq 0,
\end{equation}
a contradiction. Hence $\underline{t}(\underline{s}(m))\neq 0$ meaning that $\underline{s}(m)\in B_{11}$.
Similarly one shows that $\underline{s}$ sends $B_{01}$ to $0$ while 
$\underline{t}$ sends $B_{01}$ to $B_{11}$.

Finally, let $m\in B_{11}$ and consider $\underline{s}(m)\neq 0$. Then, again,
$\underline{s}(\underline{s}(m))=2\underline{s}(m)\neq 0$ and, moreover, 
the computation \eqref{eq721} implies $\underline{t}(\underline{s}(m))\neq 0$.
Therefore $\underline{s}(m)\in B_{11}$. Analogously one shows that $\underline{t}(m)\in B_{11}$.
The claim of the lemma follows.
\end{proof}

Note that $\sim$ is different from the full relation as $M$ is proper. Therefore $\sim$ must be
the equality relation due to the fact that $M$ is elementary. This means that $|B_{\varepsilon\delta}|\leq 1$,
for all $\varepsilon,\delta\in\{0,1\}$. If some $B_{\varepsilon\delta}$ is non-empty, we 
set $B_{\varepsilon\delta}=\{b_{\varepsilon\delta}\}$. From \eqref{eq720} it follows that
each such $b_{\varepsilon\delta}$ is an idempotent. Therefore $M$ is a sub-semi-lattice of the
following meet-semilattice, the Hasse diagram of which is depicted by the solid lines,
where the action of $\underline{s}$ is depicted by the dashed arrows while 
the action of $\underline{t}$ is depicted by the dotted arrows.
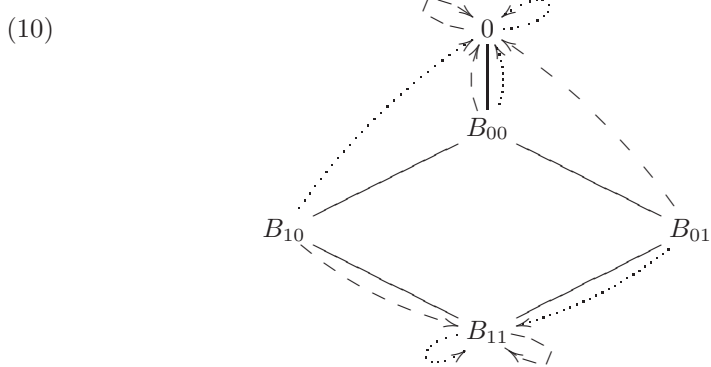
\begin{equation}\label{eq722}
\xymatrix{
&&0\ar@{-}[d]\ar@{-->}@(l,ul)[]\ar@{.>}@(r,ur)[]&&\\
&&B_{00}\ar@{-}[drr]\ar@{-}[dll]\ar@{.>}@/_/[u]\ar@{-->}@/^/[u]&&\\
B_{10}\ar@{-}[drr]\ar@{-->}@/_/[drr]\ar@{.>}@/^/[uurr]
&&&&B_{01}\ar@{-}[dll]\ar@{-->}@/_/[uull]\ar@{.>}@/^/[dll]\\
&&B_{11}\ar@{.>}@(l,dl)[]\ar@{-->}@(r,dr)[]&&\\
}
\end{equation}

Since $M$ is proper, we have $B_{00}\cup B_{10}\cup B_{01}\cup B_{11}\neq\varnothing$.
If $B_{10}\cup B_{01}\cup B_{11}=\varnothing$, then $M$ is isomorphic to $M_2$.
If $B_{10}\cup B_{01}\cup B_{11}\neq\varnothing$, then from \eqref{eq722} one sees that
$\sim_{C}$, where $C=\{0\}\cup B_{00}$, is an $R$-congruence on $M$ different from the
full relation. As $M$ is elementary, it follows that $B_{00}=\varnothing$, which 
we assume from now on.

If $B_{10}\cup B_{01}=\varnothing$, then $M$ is isomorphic to $M_3$. 
If $B_{10}\cup B_{01}\neq \varnothing$, then $B_{11}\neq\varnothing$ as $B_{11}$
contains the image of $B_{10}$ under $\underline{s}$ and the image of $B_{01}$ under $\underline{t}$.
So, if both $B_{10}$ and $B_{01}$ are non-empty, $M$ is isomorphic to $M_7$.
If $B_{01}$ is empty, $M$ is isomorphic to $M_8$.
If $B_{10}$ is empty, $M$ is isomorphic to $M_9$.
This completes the proof.
\end{proof}

\section{Finitely generated $\mathbb{Z}_{\geq 0}$-semirings}\label{s8}

\subsection{Basic structure theory}\label{s8.1}

In this section we assume that $R$ is a finitely generated $\mathbb{Z}_{\geq 0}$-semiring in the sense that
$R$ contains a finite $\mathbb{Z}_{\geq 0}$-basis $\mathbf{r}=(r_1,r_2,\dots,r_k)$. Then every element in 
$R$ can be uniquely written as a linear combination of elements in $\mathbf{r}$ with coefficients in 
$\mathbb{Z}_{\geq 0}$.

\begin{lemma}\label{lem8-1}
A finitely generated $\mathbb{Z}_{\geq 0}$-semiring contains a unique $\mathbb{Z}_{\geq 0}$-basis.
\end{lemma}

\begin{proof}
It is enough to argue that each $r_i$ must be in any  $\mathbb{Z}_{\geq 0}$-basis. For this it is enough to show
that, if we write $r_i$ as a linear combination of some (different) elements $x_1,x_2,\dots,x_m$ of $R$ with
coefficients in $\mathbb{Z}_{\geq 0}$, then $x_j=r_i$, for some $j$, the coefficient at $x_j$ is $1$
and all other coefficients are zero. 

For an element $v=\sum_i a_i r_i$, where all $a_i\in \mathbb{Z}_{\geq 0}$, we will call the number 
of non-zero $a_i$ the {\em size} of $v$ and denote it by $\mathrm{size}(v)$. Clearly, we have 
\begin{displaymath}
\mathrm{size}(v+w)\geq \max(\mathrm{size}(v),\mathrm{size}(w)).
\end{displaymath}
Furthermore, if $a\in \mathbb{Z}_{>0}$, we also have 
\begin{displaymath}
\mathrm{size}(av)=\mathrm{size}(v).
\end{displaymath}
Consequently, as $\mathrm{size}(r_i)=1$, if $r_i$ is a linear combination of $x_1,x_2,\dots,x_m$  with
coefficients in $\mathbb{Z}_{\geq 0}$ and the coefficient at some $x_j$ is non-zero, then
$\mathrm{size}(x_j)=1$ and hence $x_j=r_i$ due to the fact that $\mathbf{r}$ is a $\mathbb{Z}_{\geq 0}$-basis.
The claim follows.
\end{proof}

\begin{example}\label{ex8-2}
{\rm 
Let $A$ be a finite dimensional algebra over a field $\Bbbk$. Consider the category $A$-proj of
finitely generated projective $A$-modules. Assume that $A$-proj has the structure of a tensor category with respect to
some biadditive tensor product bifunctor $\otimes$. Then $\oplus$ and $\otimes$ induce the natural
structure of a finitely generated $\mathbb{Z}_{\geq 0}$-semiring on the set of
isomorphism classes of objects in $A$-proj.
}
\end{example}

\subsection{Cells}\label{s8.2}

This subsection adjusts \cite[Subsection 3.2]{KM} to the setup of finitely generated $\mathbb{Z}_{\geq 0}$-semirings.

Let $R$ be a finitely generated $\mathbb{Z}_{\geq 0}$-semiring and 
$\mathbf{r}=\{r_1,r_2,\dots,r_k\}$ its unique $\mathbb{Z}_{\geq 0}$-basis.
Define a partial pre-order $\leq_L$ on the set $\{r_1,r_2,\dots,r_k\}$
as follows: $r_i\leq_L r_j$ provided that there is some $r\in R$ such that 
the coefficient at $r_j$ in $rr_i$ is non-zero. The pre-order $\leq_L$
is called the {\em left} pre-order. The equivalence classes of 
$\leq_L$ are called {\em left cells}. We write $r_i\sim_L r_j$
provided that $r_i\leq_Lr_j$ and $r_j\leq_L r_i$.
Using multiplication with $r$ on the right, one similarly define 
the {\em right} pre-order $\leq_R$ and {\em right cells} corresponding
to the equivalence relation $\sim_R$.
Using multiplication with $r$ and $r'$ on both sides, one similarly define 
the {\em two-sided} pre-order $\leq_J$ and {\em two-sided cells}
corresponding to the equivalence relation $\sim_J$.

The intersection of a left and a right cell is called an {\em $\mathcal{H}$-cell},
following \cite{Gr}.

A two-sided cell $\mathcal{J}$ is called {\em strongly regular} if the
intersection of any left and any right cell inside $\mathcal{J}$
is a singleton. A strongly regular left (resp. right) cell is a left 
(resp. right) cell which belongs to a strongly regular two-sided cell.
A two-sided cell $\mathcal{J}$ is called {\em idempotent} if
it contains (not necessarily distinct) elements $x,y,z$ such that 
$xy$ has a non-zero coefficient at $z$. A two-sided cell which is
not idempotent is called {\em nilpotent}. By \cite[Corollary~19]{KM}, no two
left (or two right) cells inside an idempotent two-sided cell
can  be comparable with respect to the left (resp. right) order.

Cells are important to understand annihilators of some semimodules,
as demonstrated, for example, in the next result.

\begin{lemma}\label{lem8-21}
Let $R$ be a finitely generated $\mathbb{Z}_{\geq 0}$-semiring
and $M$ a proper $R$-se\-mi\-mo\-du\-le in which every element is idempotent.
Let $\mathcal{J}$ be a two-sided cell in $R$. Then either  
all elements in $\mathcal{J}$ annihilate $M$ or none of them does.
\end{lemma}

\begin{proof}
Assume that $r_i\cdot M=0$, for some  $r_i\in\mathcal{J}$, and let
$r_j\in\mathcal{J}$. Then there are $a,b\in R$ such that $ar_ib$
has a non-zero coefficient at $r_j$. Clearly, $ar_ib\cdot M=0$.
As the sum of a non-zero element in $M$ and any element in $M$ is non-zero,
it follows that $r_j\cdot M=0$.
\end{proof}

\subsection{Cell semimodules}\label{s8.3}

Define on $R$ an equivalence relation $\rho$ as follows:
\begin{displaymath}
\left(r=\sum_{i=1}^ka_i r_i\right)\,\rho\, 
\left(r'=\sum_{i=1}^ka'_i r_i\right)
\end{displaymath}
if and only if $a_i\neq 0$ is equivalent to $a'_i\neq 0$, for all $i$.
It is easy to see that $\rho$ is a congruence on $R$. The quotient $\tilde{R}:=R/\rho$
is a $\mathbf{B}$-semiring with the unique basis given by the $\rho$-classes
of the elements in $\mathbf{r}$. The natural projection $R\tto \tilde{R}$ is a
homomorphism of semi-rings. As explained in \cite[Subsection~2.1]{KuM}, the semiring $\tilde{R}$
can be identified with the multisemigroup structure which it induces 
on the set $\{r_1,r_2,\dots,r_k\}$. Abusing notation, we will identify 
the elements $r_i$ with their classes in $\tilde{R}$.

Let $\mathcal{L}$ be a left cell of $\mathbf{r}$. Let $C_{\mathcal{L}}$ be the additive
submonoid of $\tilde{R}$ generated by all $r_i\in\mathcal{L}$. We denote by $\pi_{\mathcal{L}}$
the natural projection of $\tilde{R}$ onto $C_{\mathcal{L}}$. Then we have
\begin{displaymath}
\pi_{\mathcal{L}}(\sum_{i=1}^ka_i r_i) =
\sum_{r_i\in\mathcal{L}}a_i r_i.
\end{displaymath}
For $r\in\tilde{R}$ and $x\in C_{\mathcal{L}}$, set 
\begin{equation}\label{eq8-3.1}
r\cdot x:=\pi_{\mathcal{L}}(rx). 
\end{equation}

\begin{lemma}\label{lem8-3}
Formula~\ref{eq8-3.1} defines on $C_{\mathcal{L}}$ the structure of an
$\tilde{R}$-semimodule.
\end{lemma}

\begin{proof}
Let $r\in\tilde{R}$ and $x\in C_{\mathcal{L}}$.
If $r_j$ appears with a non-zero coefficient in the expression of $rx$,
then $r_j\geq_L \mathcal{L}$. In case $r_j> \mathcal{L}$, we have 
$\pi_{\mathcal{L}}(r'r_j)=0$, for all $r'\in \tilde{R}$, and hence such $r'r_j$ 
has no affect on the left hand side of \ref{eq8-3.1}. The claim follows. 
\end{proof}

Pulling back via the homomorphism $R\tto \tilde{R}$, the monoid $C_{\mathcal{L}}$
becomes an $R$-semimodule. For both semirings $R$ and $\tilde{R}$, we call
the semimodule $C_{\mathcal{L}}$ the {\em cell semimodule} corresponding to $\mathcal{L}$.

This construction can be compared with similar constructions of various types
of ``cell modules'' in \cite{GM,KL,MM,KM}.

\subsection{Minimality of cell semimodules for strongly regular cells}\label{s8.5}

\begin{theorem}\label{thm8-4}
Let $R$ be a finitely generated $\mathbb{Z}_{\geq 0}$-semiring 
and $\mathcal{L}$ a left cell in $R$
contained in an idempotent strongly regular two-sided cell. 
Then the  $R$-semimodule $C_{\mathcal{L}}$ is minimal.
\end{theorem}

\begin{proof}
By passing, if necessary, to a suitable quotient of $R$,
without loss of generality, we may assume that the two-sided cell $\mathcal{J}$
of $R$ containing $\mathcal{L}$ is the maximum element with respect to the
two-sided order. As mentioned above, no two left (resp. right) cells of 
$\mathcal{J}$ are comparable with respect to the left (resp. right) order.

Let $x,y,z$ be three elements in $\mathcal{J}$ such that 
$xy=z$ in $\tilde{R}$ which exist as $\mathcal{J}$ is strongly regular and idempotent.
Then $x$ does not annihilate $C_{\mathcal{L}_y}$. Therefore, by Lemma~\ref{lem8-21},
none of the elements in $\mathcal{J}$ annihilates $C_{\mathcal{L}_y}$. In particular,
none of the elements in $\mathcal{L}$ annihilates $C_{\mathcal{L}_y}$. Hence,
for any $x'\in \mathcal{L}$, there are $y',z'\in \mathcal{L}_y$ such that 
$x'y'=z'$ in $\tilde{R}$. 

As $x'$ and $z'$ necessarily lie in the same right cell (as all right cells
inside $\mathcal{J}$
are incomparable with respect to the right order and $\mathcal{J}$ is the maximum
two-sided cell), it follows that $z'=y$ if $x'$ is in the right cell of $y$.
Therefore, in this case we have  $z=xy=xx'y'$, in particular, $xx'\neq 0$.
Consequently, none of the elements in $\mathcal{J}$ annihilates $C_{\mathcal{L}}$. 

Assume that  $\mathcal{L}=\{x_1,x_2,\dots,x_m\}$. Let $N$ be a non-zero subsemimodule
of $C_{\mathcal{L}}$ and 
\begin{displaymath}
0\neq p_1=\sum_{i=1}^ma^{(1)}_ix_i\in N. 
\end{displaymath}
Without loss of generality we may assume that $a^{(1)}_1\neq 0$. As $\mathcal{L}$ is a left cell,
acting on $p_1$ by elements from $R$ we obtain that, for each $x_j$,
our semimodule $N$ contains an element
\begin{displaymath}
0\neq p_j=\sum_{i=1}^ma^{(j)}_ix_i\in N
\end{displaymath}
such that $a^{(j)}_j\neq 0$.

Take now any fixed $x_s$. As it does not annihilate $C_{\mathcal{L}}$, there exists
$x_t$ such that $x_sx_t\neq 0$ implying $x_sx_t=x_s$ by strong regularity.
As $x_sx_{t'}=x_s$ or $0$, for any other $x_{t'}$, we obtain that
$x_s p_t=x_s\in N$. This means that $N$ contains $x_s$. Consequently, 
$N=C_{\mathcal{L}}$, as asserted.
\end{proof}

\subsection{Apex}\label{s8.6}

Let $R$ be a finitely generated $\mathbb{Z}_{\geq 0}$-semiring and
$M$ a proper minimal $R$-semimodule. By Lemma~\ref{lem8-21}, for a
given two-sided cell $\mathcal{J}$, either all or none of the elements
of $\mathcal{J}$ annihilate $M$. Let $\mathcal{J}$ be a maximal,
with respect to the two-sided order, two-sided cell which does
not annihilate $M$. Then the subsemimodule of $M$ generated by $\mathcal{J}M$
is non-zero and hence coincides with $M$ by the minimality of the latter.

If $\mathcal{J}'$ is another maximal,
with respect to the two-sided order, two-sided cell which does
not annihilate $M$, then, similarly, the subsemimodule of $M$ generated by $\mathcal{J}'M$
coincides with $M$. This implies $\mathcal{J}'\mathcal{J}M\neq 0$ and hence
$\mathcal{J}'=\mathcal{J}$ by the maximality. This unique maximal
(with respect to the two-sided order) two-sided cell which does not
annihilate $M$ is called the {\em apex} of $M$, cf. \cite{CM,KM,IRS}.
From the previous argument it follows that the apex is an
idempotent two-sided cell.

\subsection{Proper minimal semimodules from cell semimodules}\label{s8.7}

\begin{theorem}\label{thm8-5}
Let $R$ be a finitely generated $\mathbb{Z}_{\geq 0}$-semiring.
Assume that all two-sided cells in $R$ are strongly regular and idempotent. 
Then every minimal proper $R$-semimodule is a quotient of a cell semimodule.
Conversely, every quotient of a cell semimodule is minimal.
\end{theorem}

\begin{proof}
Let $M$ be a  minimal proper $R$-semimodule and $\mathcal{J}$ be its apex.
Let $\mathcal{L}$ be a left cell in $\mathcal{J}$ and $a\in\mathcal{L}$.
Then, for any $r\in R$ and $m\in M$, we either have 
\begin{displaymath}
ra\cdot m=\sum_{b\in\mathcal{L}}c_b b\cdot m,
\end{displaymath}
where all $c_b\in\mathbb{Z}_{\geq 0}$, or $ra\cdot m=0$. This is due to the combination
of the facts that $\mathcal{J}$ is the apex of $M$ and  that
$\mathcal{L}$ is not left comparable to any other left cell in $\mathcal{J}$. 

Now, take $m\in M$ such that $a\cdot m\neq 0$. Then the
map $R\ni x\mapsto x\cdot m$ is a homomorphism of $R$-semimodules which,
using the facts that $M$ consist of idempotents (see Corollary~\ref{corminprosem}),
induces a non-zero homomorphism from $C_{\mathcal{L}}$ to $M$.
By the minimality of $M$, the latter homomorphism is surjective. The 
first claim follows. The second claim follows from Theorem~\ref{thm8-4} 
and Proposition~\ref{prop2.2}.
\end{proof}

\begin{corollary}\label{cor8-51}
Let $R$ be a finitely generated $\mathbb{Z}_{\geq 0}$-semiring
and $\mathcal{J}$ a strongly regular and idempotent two-sided cell in $R$. 
Let $\mathcal{L}$ and $\mathcal{L}'$ be two left cells  in $\mathcal{J}$.
Then $C_{\mathcal{L}}\cong C_{\mathcal{L}'}$.
\end{corollary}

\begin{proof}
Both $C_{\mathcal{L}}$ and $C_{\mathcal{L}'}$ are minimal
(cf. Theorem~\ref{thm8-4}), proper and
have apex $\mathcal{J}$. From the proof of Theorem~\ref{thm8-5} it follows
that there is a surjective homomorphism  $\varphi:C_{\mathcal{L}}\tto C_{\mathcal{L}'}$.

As $\mathcal{J}$ is strongly regular, all left cells in $\mathcal{J}$
have the same cardinality (the number of right cells in $\mathcal{J}$). 
As both $C_{\mathcal{L}}$ and $C_{\mathcal{L}'}$ are finite of
respective cardinalities $2^{|\mathcal{L}|}$ and 
$2^{|\mathcal{L}'|}$, it follows that $\varphi$ is an isomorphism.
\end{proof}

\subsection{Reduced cell semimodules}\label{s8.8}

From now on, for simplicity, we assume that all two-sided cells of $R$ are idempotent.

Let $\mathcal{L}$ be a left cell of $R$. Let $\mathcal{H}_i$, where $i\in I$, be a complete 
list of non-empty $\mathcal{H}$-cells in $\mathcal{L}$. Consider the boolean $\tilde{C}_{\mathcal{L}}:=2^{I}$
which has the natural structure of a commutative monoid under the boolean addition.

For each $r_j\in\mathbf{r}$ and $i\in I$, we define $r_j\cdot i$ as the set of all elements $s\in I$
for which there exist $x\in \mathcal{H}_i$ and $y\in \mathcal{H}_s$ such that $y$ appears with a 
non-zero coefficient in $r_jx$. 

\begin{proposition}\label{prop8-81}
This defines on  $\tilde{C}_{\mathcal{L}}$ the structure of a minimal $R$-semimodule, moreover,
$\tilde{C}_{\mathcal{L}}$ is a quotient of ${C}_{\mathcal{L}}$.
\end{proposition}

\begin{proof}
We define on  ${C}_{\mathcal{L}}$ an equivalence relation $\tau$ as follows: two elements $x$ and $y$ of 
${C}_{\mathcal{L}}$ are $\tau$-equivalent if and only if, for each $i\in I$, some element of
$\mathcal{H}_i$ appears with a non-zero coefficient in $x$ if and only if some
element of $\mathcal{H}_i$  (but not necessarily the same element as for $x$) 
appears with a non-zero coefficient in $y$.

The underlying monoid of ${C}_{\mathcal{L}}$ is isomorphic to the boolean of $\mathcal{L}$ with 
respect to the operation of boolean addition. The equivalence relation $\tau$ on ${C}_{\mathcal{L}}$
is generated by the  equivalence relation on $\mathcal{L}$ with equivalence classes $\mathcal{H}_i$, 
where $i\in I$. Therefore $\tau$ is a congruence on the underlying monoid of ${C}_{\mathcal{L}}$
and the quotient ${C}_{\mathcal{L}}/\tau$ is canonically isomorphic to $\tilde{C}_{\mathcal{L}}$. 
We claim that $\tau$ is even an $R$-congruence. To prove this, let $a$ and $b$ be two elements
in some $\mathcal{H}_i$. We need to show that $r_j a$ and $r_j b$ are $\tau$-equivalent.

Let $\mathcal{J}$ be the two-sided cell containing $\mathcal{L}$. By our assumptions,
$\mathcal{J}$ is idempotent, in particular, by \cite[Corollary~19]{KM}, any two different
right cells in $\mathcal{J}$ are incomparable with respect to the right order.
Therefore the fact that $r_j a$ and $r_j b$ are $\tau$-equivalent is equivalent to the
fact that the $\leq_R$-ideals generated by $r_j a$ and $r_j b$ coincide. As $a$ and $b$
are in the same $\mathcal{H}$-cell, the  $\leq_R$-ideals generated by $a$ and $b$ coincide.
Hence the $\leq_R$-ideals generated by $r_j a$ and $r_j b$ coincide as well. 

The above implies that $\tau$ is an $R$-congruence on  ${C}_{\mathcal{L}}$. 
From the definition of $\tau$ it follows directly that the quotient 
${C}_{\mathcal{L}}/\tau$ is canonically isomorphic to 
$\tilde{C}_{\mathcal{L}}$ also as an $R$-semimodule. 

That $\tilde{C}_{\mathcal{L}}$ is minimal is proved similarly to the proof of Theorem~\ref{thm8-4}.
\end{proof}

The semimodule $\tilde{C}_{\mathcal{L}}$ will be called the {\em reduced cell semimodule}
corresponding to  $\mathcal{L}$. From the proof of Proposition~\ref{prop8-81} it follows that, in case
$\mathcal{L}$ belongs to a strongly regular two-sided cell, we have $\tilde{C}_{\mathcal{L}}\cong
{C}_{\mathcal{L}}$.

\subsection{Proper minimal semimodules from reduced cell semimodules}\label{s8.9}

\begin{theorem}\label{thm8-91}
Let $R$ be a finitely generated $\mathbb{Z}_{\geq 0}$-semiring.
Assume that all two-sided cells in $R$ are idempotent. 
Then every minimal proper $R$-semimodule is a quotient of a reduced cell semimodule.
Conversely, every quotient of a reduced cell semimodule is minimal.
\end{theorem}

\begin{proof}
We outline an argument which is similar to the proof of Theorem~\ref{thm8-5}.
Let $M$ be a minimal proper $R$-semimodule, 
$\mathcal{J}$ the apex of $M$ and $\mathcal{L}$ a left cell in $\mathcal{J}$.
Note that each element in $M$ is idempotent by Corollary~\ref{corminprosem}.
Let $m\in M$ be a non-zero element which is not annihilated by some element in 
$\mathcal{L}$. The map $x\mapsto x\cdot m$, from $R$ to $M$, is a homomorphism of
$R$-semimodules which induces a non-zero homomorphism from $C_{\mathcal{L}}$ to
$M$. By the minimality of $M$, we obtain that $M$ is a quotient of $C_{\mathcal{L}}$.

We claim that the quotient map factors through $\tilde{C}_{\mathcal{L}}$. For each
$\mathcal{H}_i$, let $h_i\in R$ denote the sum of all elements in $\mathcal{H}_i$.
Consider the submonoid  $N'$ of $M$ generated by $h_i\cdot m$, $i\in I$,
and the subsemimodule $N=RN'$. Clearly, $N'\subset N$. 

We claim that $N=N'$. For this, we need to show that each $r_j\cdot(h_i\cdot m)$ equals
the sum of all $h_s\cdot m$, for which $r_jh_i$ contains an element in $\mathcal{H}_s$
with a non-zero coefficient. Let $s$ be such that $r_jh_i$ contains an element in 
$\mathcal{H}_s$ with a non-zero coefficient. We need to show that $r_jh_i$ has a
non-zero coefficient at each element from $\mathcal{H}_s$. Consider the right ideal
$r_jh_iR$. By assumptions, $r_jh_iR$ contains an element with a non-zero coefficient 
at some element in $\mathcal{H}_s$. Since $r_jh_iR$ is a right ideal, every element
in $\mathcal{H}_s$ has to appear with a non-zero coefficient in some element of 
$r_jh_iR$. At the same time, by \cite[Corollary~19]{KM}, the facts that
$r_p\geq_{\mathcal{R}}\mathcal{H}_i$ and $r_jr_p$ has a
non-zero coefficient at each element from $\mathcal{H}_s$ imply $r_p\in \mathcal{H}_i$.
Therefore $r_jh_i$ must contain all elements in $\mathcal{H}_s$ with non-zero coefficient.
Consequently, $N=N'$.

By the minimality of $M$, we have $M=N$. Therefore, mapping $i$ to $h_i\cdot m$, for $i\in I$,
extends to an epimorphism from $\tilde{C}_{\mathcal{L}}$ to $N$. The first part of 
the claim follows. The second part is proved similarly to Theorem~\ref{thm8-5}.
\end{proof}

\section{The Kazhdan-Lusztig semiring of a dihedral group}\label{s9}

\subsection{Dihedral groups, their group algebras and the corresponding Kazh\-dan-Lusztig semirings}\label{s9.1}

In this section, we fix $n\geq 3$. Let $D_{2\cdot n}$ denote the {\em dihedral group} of
symmetries of a regular $n$-gon in a plane. The group $D_{2\cdot n}$ is a Coxeter group
with presentation
\begin{displaymath}
D_{2\cdot n}=\langle s,t\,:\, s^2=t^2=(st)^n=(ts)^n=e\rangle.
\end{displaymath}
The group $D_{2\cdot n}$ has $2n$ elements
\begin{displaymath}
D_{2\cdot n}=\{e,s,t,st,ts,sts,tst,\dots,w_0:=\underbrace{stst\dots}_{n\text{ factors}}=
\underbrace{tsts\dots}_{n\text{ factors}}\}.
\end{displaymath}
The {\em Bruhat order} on $D_{2\cdot n}$, denoted $\preceq$, has the following Hasse diagram
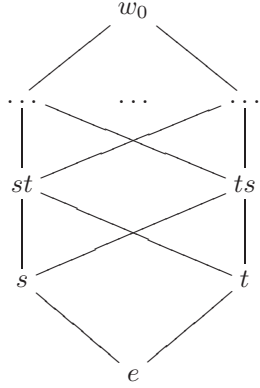
\begin{displaymath}
\xymatrix{ 
&w_0\ar@{-}[dl]\ar@{-}[dr]&\\
\dots\ar@{-}[drr]\ar@{-}[d]&\dots&\dots\ar@{-}[dll]\ar@{-}[d]\\
st\ar@{-}[drr]\ar@{-}[d]&&ts\ar@{-}[dll]\ar@{-}[d]\\
s\ar@{-}[dr]&&t\ar@{-}[dl]\\
&e&
}
\end{displaymath}

Consider the integral group ring $\mathbb{Z}[D_{2\cdot n}]$ with the standard basis consisting of
group elements. For $w\in D_{2\cdot n}$, define the corresponding {\em Kazhdan-Lusztig basis} element
$\underline{w}$ as follows:
\begin{displaymath}
\underline{w}:=\sum_{x\preceq w}x.
\end{displaymath}
Then $\{\underline{w}:w\in D_{2\cdot n}\}$ is another basis of $\mathbb{Z}[D_{2\cdot n}]$, called the
{\em Kazhdan-Lusztig basis}.
In this section, our main object of study is the $\mathbb{Z}_{\geq 0}$-subsemiring $R$
of $\mathbb{Z}[D_{2\cdot n}]$
generated by the Kazhdan-Lusztig basis elements. The latter elements form the unique 
$\mathbb{Z}_{\geq 0}$-basis in the $\mathbb{Z}_{\geq 0}$-semiring $R$. 

The semiring $R$ has three two-sided cells: 
\begin{itemize}
\item the cell $\{\underline{e}\}$ which is, at the same time, a left and a right cell;
\item the cell $\{\underline{w_0}\}$ which is, at the same time, a left and a right cell;
\item the cell, which we denote by $\mathcal{J}$, that consists of all remaining elements.
\end{itemize}
All these two-sided cells are idempotent. The cell $\mathcal{J}$ consists of two left cells:
\begin{itemize}
\item the left cell $\mathcal{L}_{s}$ containing $\underline{s}$, it consists of all $\underline{w}$
such that $w\neq w_0$ and the unique reduced expression of $w$ has $s$ as the rightmost letter;
\item the left cell $\mathcal{L}_{t}$ containing $\underline{t}$, it consists of all $\underline{w}$
such that $w\neq w_0$ and the unique reduced expression of $w$ has $t$ as the rightmost letter.
\end{itemize}
The cell $\mathcal{J}$ consists of two right cells:
\begin{itemize}
\item the right cell $\mathcal{R}_{s}$ containing $\underline{s}$, it consists of all $\underline{w}$
such that $w\neq w_0$ and the unique reduced expression of $w$ has $s$ as the leftmost letter;
\item the right cell $\mathcal{R}_{t}$ containing $\underline{t}$, it consists of all $\underline{w}$
such that $w\neq w_0$ and the unique reduced expression of $w$ has $t$ as the leftmost letter.
\end{itemize}
Consequently, $\mathcal{L}_{s}$ consists of two $\mathcal{H}$-cells: 
$\mathcal{L}_{s}\cap \mathcal{R}_{s}$ and $\mathcal{L}_{s}\cap \mathcal{R}_{t}$,
and similarly for all other left and right cells in $\mathcal{J}$.

Recall the following formulae:
\begin{equation}\label{eqnkls}
\underline{s}\cdot \underline{w}=
\begin{cases}
\underline{sw}, & w=e,t;\\
2\underline{w}, & \text{$w$ has a reduced expression of the form $s\dots$};\\
\underline{sw}+\underline{tw}, & \text{else};
\end{cases}
\end{equation}
and
\begin{equation}\label{eqnklt}
\underline{t}\cdot \underline{w}=
\begin{cases}
\underline{tw}, & w=e,s;\\
2\underline{w}, & \text{$w$ has a reduced expression of the form $t\dots$};\\
\underline{sw}+\underline{tw}, & \text{else}.
\end{cases}
\end{equation}

If $n=3$, then our semiring $R$ coincides with the semiring $R$ considered in Subsection~\ref{s7.1}.

For more details on Kazhdan-Lusztig combinatorics of dihedral groups, we refer the reader to \cite{El}.

\subsection{Classification of minimal proper $R$-semimodules}\label{s9.2}

Consider the $R$-se\-mi\-mo\-du\-le $\tilde{C}_{\mathcal{L}_{s}}$. We have $\tilde{C}_{\mathcal{L}_{s}}:=2^{\{x,y\}}$,
where $x$ corresponds to $\mathcal{L}_{s}\cap \mathcal{R}_{s}$ and $y$
corresponds to $\mathcal{L}_{s}\cap \mathcal{R}_{t}$. The action  of $R$ on $x$ and $y$ is given,
for $w\in D_{2\cdot n}$, by
\begin{displaymath}
\underline{w}\cdot x=
\begin{cases}
0, & w=w_0,\\ 
y, & w\in \mathcal{R}_{t},\\
x, & \text{else};
\end{cases}
\qquad
\underline{w}\cdot y=
\begin{cases}
0, & w=w_0,\\ 
x, & w\in \mathcal{R}_{s},\\
y, & \text{else}.
\end{cases}
\end{displaymath}

\begin{lemma}\label{lem9-01}
The $R$-semimodule  $\tilde{C}_{\mathcal{L}_{s}}$ has exactly three non-trivial quotients, namely,
\begin{itemize}
\item the $3$-element quotient $N_1$ in which $x+y$ is identified with $x$;
\item the $3$-element quotient $N_2$ in which $x+y$ is identified with $y$;
\item the $2$-element quotient $N_3$ in which $x+y$ is identified with both $x$ and $y$.
\end{itemize}
\end{lemma}

\begin{proof}
The $R$-semimodule  $\tilde{C}_{\mathcal{L}_{s}}$ can be depicted using the left picture in \eqref{eqnewfor9},
where $x$ corresponds to $(1,0)$ and $y$ to $(0,1)$, the action of elements in $\mathcal{R}_{s}$
is given by the dashed arrows and the action of elements in $\mathcal{R}_{t}$
is given by the dotted arrows. The claim of the lemma is then checked similarly to Subsection~\ref{s7.5}.
\end{proof}

As an immediate corollary from Theorem~\ref{thm8-91} and Lemma~\ref{lem9-01}, we have
the following claim (cf. Theorem~\ref{thm7.2}).

\begin{corollary}\label{cor9-02}
The $R$-semimodules ${C}_{\{\underline{e}\}}$, ${C}_{\{\underline{w_0}\}}$,
$\tilde{C}_{\mathcal{L}_{s}}$, $N_1$, $N_2$ and $N_3$ are the only minimal proper $R$-semimodules. 
\end{corollary}

\subsection{Classification of elementary proper $R$-semimodules}\label{s9.3}

Inspired by the fact, established in Corollary~\ref{cor9-02}, that our classification of 
minimal proper semimodules can be extended from the case $D_{2\cdot 3}$ to all dihedral cases,
it is natural to ask whether the same can be done about classification of elementary proper semimodules.
This is the aim of this subsection.

Let $K$ denote the boolean $2^{\{x,y\}}$ which we consider as an abelian monoid with respect to the
boolean addition. We define on $K$ the structure of an $R$-semimodule as follows:
\begin{itemize}
\item $\underline{w_0}$ acts on $K$ as zero;
\item each element $\underline{w}$ in $\mathcal{L}_s$ annihilates $y$ and maps both $x$ and $x+y$ to $x+y$;
\item each element $\underline{w}$ in $\mathcal{L}_t$ annihilates $x$ and maps both $y$ and $x+y$ to $x+y$.
\end{itemize}
It is straightforward to check that this defines on $K$ the structure of an $R$-semimodule.
We denote by
\begin{itemize}
\item $K_1$ the subsemimodule of $K$ consisting of $0$, $x$ and $x+y$;
\item $K_2$ the subsemimodule of $K$ consisting of $0$, $y$ and $x+y$;
\item $K_3$ the subsemimodule of $K$ consisting of $0$ and $x+y$.
\end{itemize}
The $R$-semimodule $K$ can be depicted using the right picture in \eqref{eqnewfor9}
where $x$ corresponds to $(1,0)$ and $y$ to $(0,1)$, the action of elements in $\mathcal{L}_{s}$
is given by the dashed arrows and the action of elements in $\mathcal{L}_{t}$
is given by the dotted arrows. 

\begin{theorem}\label{thm9-03}
The $R$-semimodules ${C}_{\{\underline{e}\}}$, ${C}_{\{\underline{w_0}\}}$,
$K$, $K_1$, $K_2$ and $K_3$ are the only elementary proper $R$-semimodules.
\end{theorem}

\begin{proof}
It is straightforward to check that all $R$-semimodules in the formulation are elementary.
To complete the proof one needs to check that there are no other elementary proper $R$-semimodules.
This is done similarly to the proof of Theorem~\ref{thm7.3}. The only non-trivial part 
is to prove an analogue of Lemma~\ref{lem7.11}.

So, we assume that $M$ is a proper elementary $R$-semimodule such that $\underline{w_0}M=0$.
By \cite[Proposition~1.2]{Il}, all elements of $M$ are idempotent. Let
$B_{00}$, $B_{10}$, $B_{01}$ and $B_{11}$ be defined as in Lemma~\ref{lem7.11}. We claim that 
the corresponding equivalence relation $\sim$ correspondning to the one 
in Lemma~\ref{lem7.11} is an  $R$-congruence on $M$.
That $\sim$ is a congruence is checked similarly to Lemma~\ref{lem7.11}, so 
we just need to prove that $\sim$ is $R$-invariant.

We will need the following lemma.

\begin{lemma}\label{lem9-04}
Let $r\in\{s,t\}$. Then, for $m\in M$, the conditions
\begin{enumerate}[$($a$)$]
\item\label{lem9-04.1} $m$ is annihilated by some element of $\mathcal{L}_r$,
\item\label{lem9-04.2} $m$ is annihilated by all elements of $\mathcal{L}_r$,
\end{enumerate}
are equivalent.
\end{lemma}

\begin{proof}
We only need to prove the implication \eqref{lem9-04.1}$\Rightarrow$\eqref{lem9-04.2}.
Let $m\in M$ be annihilated by some $\underline{w}\in \mathcal{L}_r$. Then $m$
is annihilated by all elements in $R\underline{w}$. As non-zero elements of $M$ form an ideal
with respect to addition,
it follows that $m$ is annihilated by any $\underline{u}$ which appears in some element
in $R\underline{w}$ with a non-zero coefficient. From the definition of $\mathcal{L}_r$
it thus follows that each $\underline{u}\in \mathcal{L}_r$ must annihilate $m$.
\end{proof}

From Lemma~\ref{lem9-04} it follows that each $\underline{w}\in \mathcal{L}_s$ sends
$B_{00}\cup B_{01}$ to zero and $B_{10}\cup B_{11}$ to something non-zero.
Let $A$ be the image of $B_{10}\cup B_{11}$ under our $\underline{w}$.
Assume $\underline{s}a=0$, for some $a\in A$, say $a=\underline{w}b$. 
If $\underline{w}\in \mathcal{R}_s$, then from \eqref{eqnkls} it follows that 
\begin{displaymath}
0= \underline{s}a= \underline{s}\cdot \underline{w}b= 2\underline{w}b=2a,
\end{displaymath}
that is $a=0$, a contradiction to the fact that all elements in $A$ are non-zero. 
If $\underline{w}\in \mathcal{R}_t$, then from \eqref{eqnkls} it similarly follows that 
$\underline{tw}b=0$. As $\underline{tw}\in \mathcal{L}_s$, this again contradicts
Lemma~\ref{lem9-04}. Therefore $0\not\in \underline{s}A$.

Assume $\underline{t}a=0$, for some $a\in A$, say $a=\underline{w}b$. 
If $\underline{w}\in \mathcal{R}_t$, then from \eqref{eqnklt} it follows that 
$a=0$, a contradiction to the fact that all elements in $A$ are non-zero. 
If $\underline{w}\in \mathcal{R}_s$, then from \eqref{eqnklt} it follows that 
$\underline{tw}b=0$ and even $\underline{sw}b=0$, if $w\neq s$. 
If $w=s$, then  $\underline{tw}\in \mathcal{L}_s$ and we get a contradiction to
Lemma~\ref{lem9-04}. If $w\neq s$, then  $\underline{sw}\in \mathcal{L}_s$ 
and we get a contradiction to Lemma~\ref{lem9-04}. 
Therefore $0\not\in \underline{t}A$. Consequently, $A\subset B_{11}$.
This implies that $\sim$ is stable under the action of any $\underline{w}\in \mathcal{L}_s$.
By symmetry, $\sim$ is also stable under the action of any $\underline{w}\in \mathcal{L}_t$.
This proves that $\sim$ is $R$-stable.

The rest of the proof of Theorem~\ref{thm9-03} is similar 
to the proof of Theorem~\ref{thm7.3} and is left to the reader.
\end{proof}

\vspace{5mm}

\noindent
Department of Mathematics, Uppsala University, Box. 480,
SE-75106, Uppsala, SWEDEN, emails: 

{\tt Chih-Whi.Chen\symbol{64}math.uu.se}

{\tt Brendan.Frisk.Dubsky\symbol{64}math.uu.se}

{\tt Helena.Jonsson\symbol{64}math.uu.se}

{\tt Volodymyr.Mazorchuk\symbol{64}math.uu.se}

{\tt Elin.Persson.Westin\symbol{64}math.uu.se}

{\tt Xiaoting.Zhang\symbol{64}math.uu.se}

{\tt Jakob.Zimmermann\symbol{64}math.uu.se}


\begin{thebibliography}{999999}
\bibitem[CM]{CM} A.~Chan, V.~Mazorchuk. Diagrams and discrete extensions for finitary $2$-representations.
Preprint arXiv:1601.00080, to appear in Math. Proc. Cambr. Phil. Soc.
\bibitem[El]{El} B.~Elias. The two-color Soergel calculus. Compos. Math. {\bf 152} (2016), no. 2, 327--398. 
\bibitem[EW]{EW} B.~Elias, G.~Williamson. The Hodge theory of Soergel bimodules. 
Ann. of Math. (2) {\bf 180} (2014), no. 3, 1089--1136. 
\bibitem[Fo]{Fo} L.~Forsberg. Multisemigroups with multiplicities and complete ordered semi-rings. 
Beitr. Algebra Geom. {\bf 58} (2017), no. 2, 405--426.
\bibitem[GaMa]{GaMa} O.~Ganyushkin, V.~Mazorchuk. Classical finite transformation semigroups. 
An introduction. Algebra and Applications, {\bf 9}. Springer-Verlag London, Ltd., London, 2009. 
\bibitem[Go]{Go} J.~Golan. Semirings and their applications. Kluwer Academic Publishers, Dordrecht, 1999.
\bibitem[GoMi]{GM} M.~Gondran, M.~Minoux. Graphs, dioids and semirings. New models and algorithms. 
Operations Research$/$Computer Science Interfaces Series, {\bf 41}. Springer, New York, 2008. 
\bibitem[Gr]{Gr} J.~A.~Green. On the structure of semigroups. Ann. of Math. (2) {\bf 54} (1951). 163--172.
\bibitem[Ha]{Ha}  P.~Hall. On the finiteness of certain soluble groups. Proc. Lond. Math. 
Soc. (3) {\bf 9} (1959), 595--622.
\bibitem[Il]{Il} S.~Il'in. $V$-semirings. (Russian) Sibirsk. Mat. Zh. {\bf 53} 
(2012), no. 2, 277--290; translation in Sib. Math. J. {\bf 53} (2012), no. 2, 222--231.
\bibitem[IRS]{IRS} Z.~Izhakian, J.~Rhodes, B.~Steinberg. Representation theory of finite 
semigroups over semirings. J. Algebra {\bf 336} (2011), 139--157.
\bibitem[KNZ]{KNZ} Y.~Katsov, T.~Nam, J.~Zumbr{\"a}gel. On simpleness of semirings and 
complete semirings. J. Algebra Appl. {\bf 13} (2014), no. 6, 1450015, 29 pp.
\bibitem[KL]{KL} D.~Kazhdan, G.~Lusztig. Representations of Coxeter groups and Hecke algebras. 
Invent. Math. {\bf 53} (1979), no. 2, 165--184. 
\bibitem[KiM]{KM} T.~Kildetoft, V.~Mazorchuk. Special modules over positively based algebras. 
Doc. Math. {\bf 21} (2016), 1171--1192.
\bibitem[KuM]{KuM} G.~Kudryavtseva, V.~Mazorchuk. On multisemigroups. Port. Math. 
{\bf 72} (2015), no. 1, 47--80.
\bibitem[Ma1]{Ma1} V.~Mazorchuk. Lectures on algebraic categorification. 
QGM Master Class Series. European Mathematical Society (EMS), Z{\"u}rich, 2012.
\bibitem[Ma2]{Ma2} V.~Mazorchuk.  Classification problems in 2-representation theory. 
S{\~a}o Paulo J. Math. Sci. {\bf 11} (2017), no. 1, 1--22. 
\bibitem[MM]{MM} V.~Mazorchuk, V.~Miemietz. Cell $2$-representations of finitary 
$2$-categories. Compos. Math. {\bf 147} (2011), no. 5, 1519--1545.
\bibitem[Wi]{Wi} D.~Wilding. Linear algebra over semirings. PhD Thesis. University of Manchester.
2014
\end{thebibliography}
\end{document}